\documentclass[11pt,a4paper]{article}

\title{Forward-backward algorithms devised by graphs%
\thanks{The authors were partially supported by Grant PID2022-136399NB-C21 funded by ERDF/EU and by MICIU/AEI/10.13039/501100011033. C.L.P. was supported by Grant PREP2022-000118 funded by MICIU/AEI/10.13039/501100011033 and by ``ESF Investing in your future''.}
}

\date{}

\newcommand*\samethanks[1][\value{footnote}]{\footnotemark[#1]}

\author{Francisco J. Arag\'on-Artacho\thanks{Department of Mathematics,
University of Alicante, \textsc{Spain}. e-mail:~\href{mailto:francisco.aragon@ua.es}{francisco.aragon@ua.es}, \href{mailto:ruben.campoy@ua.es}{ruben.campoy@ua.es}, \href{mailto:cesar.lopez@ua.es}{cesar.lopez@ua.es}}
        \and Rub\'en Campoy\samethanks[2]
        \and C\'esar L\'opez-Pastor\samethanks[2]
}

\usepackage{amsmath,amssymb,amsthm}

\newtheorem{theorem}{Theorem}[section]
\newtheorem{definition}[theorem]{Definition}
\newtheorem{proposition}[theorem]{Proposition}

\newtheorem{lemma}[theorem]{Lemma}
\newtheorem{fact[theorem]}{Fact}
\newtheorem{remark}[theorem]{Remark}
\newtheorem{example}[theorem]{Example}
\newtheorem{assumption}[theorem]{Assumption}

\usepackage{centernot}
\usepackage[margin=2.7cm]{geometry}
\usepackage[colorlinks=true,linkcolor=blue,citecolor=blue,urlcolor=blue,hypertexnames=false]{hyperref}

\usepackage{graphicx,fancybox}

\usepackage{subcaption}
\usepackage[capitalise]{cleveref}
\usepackage{mathtools}
\usepackage{pgf}
\usepackage[font={small}]{caption}

\usepackage{xcolor}

\usepackage[sortcites,style=trad-abbrv]{biblatex}
\usepackage{algorithm}
\usepackage{algpseudocodex}
\usepackage{stmaryrd}
\usepackage{todonotes}
\usepackage{enumitem}

\usepackage{tikz}
\usepackage{tikz-cd}
\usetikzlibrary{shapes}
\usepackage{adjustbox}

\usepackage{arydshln}
\usepackage[labelformat=simple]{subcaption}

\DeclareMathAlphabet\mbc{OMS}{cmsy}{b}{n}

\newcommand{\zer}{\operatorname{zer}}

\newcommand{\rk}{\operatorname{rank}}

\newcommand{\Id}{\operatorname{Id}}

\newcommand{\wto}{\rightharpoonup}

\newcommand{\Hi}{\mathcal{H}}

\newcommand{\R}{\mathbb{R}}

\newcommand{\pre}{p}

\newcommand{\norm}[1]{\left\lVert #1\right\rVert}

\let\texdisplaystyle\displaystyle
\def\displaytotextstyle{\textstyle\let\displaystyle\texdisplaystyle}
\newenvironment{taligned}
{\let\displaystyle\displaytotextstyle\aligned}
{\endaligned}

\def\tto{\rightrightarrows}
\Crefname{fact}{Fact}{Facts}
\Crefname{enumi}{}{}

\let\epsilon\varepsilon

\begin{document}

\maketitle

\begin{abstract}
    In this work, we present a methodology for devising forward-backward methods for finding zeros in the sum of a finite number of maximally monotone operators. We extend the framework and techniques from~[\emph{SIAM J. Optim.}, 34 (2024), pp. 1569--1594] to cover the case involving a finite number of cocoercive operators, which should be directly evaluated instead of computing their resolvent. The algorithms are induced by three graphs that determine how the algorithm variables interact with each other and how they are combined to compute each resolvent. The hypotheses on these graphs ensure that the algorithms obtained have minimal lifting and are frugal, meaning that the ambient space of the underlying fixed point operator has minimal dimension and that each resolvent and each cocoercive operator is evaluated only once per iteration. This framework not only allows to recover some known methods, but also to generate new ones, as the forward-backward algorithm induced by a complete graph. We conclude with a numerical experiment showing how the choice of graphs influences the performance of the algorithms.
\end{abstract}



\section{Introduction}
In this work we are interested in developing algorithms for solving structured monotone inclusion problems of the form
\begin{equation}\label{eq:Prob}
\text{find } x\in\Hi \text{ such that } 0\in\left( \sum_{i=1}^n A_i + \sum_{i=1}^m B_i \right)(x),
\end{equation}
where $A_1,\ldots,A_n:\Hi\tto\Hi$ are (set-valued) maximally monotone operators on a Hilbert space~$\Hi$, while the single-valued operators $B_1,\ldots,B_m:\Hi\to\Hi$ are cocoercive (we write $m=0$ if there are no cocoercive operators). When the sum is itself maximally monotone,  inclusion~\eqref{eq:Prob} can be tackled with the \emph{proximal point algorithm} \cite{RockaProx} (see Section~\ref{sec:PPPA}).
However, this approach becomes impractical since the resolvent of the sum  is usually not computable.

Splitting algorithms of forward-backward-type are so called because they take advantage of the structure of \eqref{eq:Prob}, establishing an iterative process that only requires the computation of individual resolvents of the maximally monotone operators $A_1,\ldots,A_n$ (backward steps) and direct evaluations of $B_1,\ldots,B_m$ (forward steps), combined by vector additions and scalar multiplications. If each resolvent and each cocoercive operator is computed exactly once per iteration, then the algorithm is said to be a \emph{frugal resolvent splitting}, a terminology introduced by Ryu in his seminal work~\cite{ryu20}. For instance, a frugal resolvent splitting for \eqref{eq:Prob} when $n=2$ and $m=1$ is the \emph{Davis--Yin splitting algorithm}~\cite{davisyin}, whose iterations take the form
\begin{equation}\label{eq:DavisYin}
\left\{ \begin{aligned}
x_1^{k+1} &=   J_{\gamma A_1}\left(w^k\right),\\
x_2^{k+1} &=   J_{\gamma A_2}\left( 2x_1^{k+1}-w^k-\gamma B(x_1^{k+1}) \right),\\
w^{k+1} &=   w^k + \theta_k \left(x_2^{k+1}-x_1^{k+1}\right), \\
\end{aligned}\right.
\end{equation}
for some starting point $w^0\in\Hi$ and $k= 0,1,\ldots$, where the positive scalars $\gamma$ and $\theta_k$ are some appropriately chosen parameters.
This method encompasses the forward-backward method (when $A_1=0$) and the backward-forward method \cite{backforw} (when $A_2=0$), which are both frugal resolvent splittings for \eqref{eq:Prob} when
$n=m=1$, as well as the Douglas--Rachford algorithm~\cite{LM79} (when $B=0$), for $n=2$ and $m=0$. Note that, although iteration \eqref{eq:DavisYin} is described by three variables, we can discern two
distinct types among them. On the one hand, only $w^k$ needs to be stored to compute the next iterate, so we will refer to it as a \emph{governing variable}. In contrast, the sequences $x_1^k$ and $x_2^k$,
which are obtained via resolvent computations, are precisely the ones that converge to a solution to~\eqref{eq:Prob}. We will refer to them as \emph{resolvent variables}.

In scenarios involving $n$ maximally monotone operators and $m=1$ cocoercive operator, we can employ the \emph{generalized forward-backward algorithm} \cite{genFB}, which iterates as
\begin{equation}\label{eq:genFB}
\left\{ \begin{taligned}
x_i^{k+1} &=   J_{\gamma A_i}\left(\tfrac{2}{n}\sum_{j=1}^n w_j^k-w_i^k-\tfrac{\gamma}{n} B(\tfrac{1}{n}\sum_{j=1}^n w_j^k)\right), \quad \forall i\in \llbracket 1, n\rrbracket,\\
w_i^{k+1} &=   w_i^k + \theta_k \left(x_i^{k+1}-\tfrac{1}{n}\sum_{j=1}^n w_j^k\right), \qquad \forall i\in \llbracket 1, n\rrbracket,
\end{taligned}\right.
\end{equation}
where $\llbracket 1,n\rrbracket=\{1,2\ldots,n\}$.
This algorithmic scheme can be deduced by applying the Davis--Yin algorithm to an adequate reformulation of \eqref{eq:Prob} in a product space~\cite{pierra}. Note that, however, the generalized forward-backward algorithm does not encompass the Davis--Yin method. This distinction is evident in the number of governing variables of each algorithm, which is termed as \emph{lifting}. While algorithm \eqref{eq:genFB} has $n$-fold lifting, due to the need of storing the value of the $n$ (governing) variables $w_1^k,\ldots,w_n^k$ at each iteration, Davis--Yin has $1$-fold lifting, as only $w^k$ in \eqref{eq:DavisYin} is required to be saved. Generally, a reduction in lifting may be preferred as it results in computational memory savings.

The notion of lifting also traces back to the work of Ryu~\cite{ryu20}, who proved that for three maximally monotone operators (i.e., problem \eqref{eq:Prob} for $n=3$ and $m=0$) the minimal lifting is $2$. This result was later generalized by Malitsky--Tam~\cite{malitsky2023resolvent} for an arbitrary number $n$ of maximally monotone operators, establishing a minimal lifting of $n-1$. As a generalization of the algorithm introduced in~\cite{malitsky2023resolvent}, the authors of~\cite{ring-networks} proposed the forward-backward-type algorithm given by
\begin{equation}\label{eq:ringFB}
\left\{ \begin{aligned}
x_1^{k+1} &=  J_{\gamma A_1}\left(w_1^k\right),\\
x_i^{k+1} &=   J_{\gamma A_i}\left( x_{i-1}^{k+1}+ w_i^k -w_{i-1}^k-\gamma B_{i-1}( x_{i-1}^{k+1}) \right), \quad \forall i\in \llbracket 2, n-1\rrbracket,\\
x_n^{k+1} &=  J_{\gamma A_n}\left( x_{1}^{k+1}+x_{n-1}^{k+1}-w_{n-1}^k-\gamma B_{n-1}( x_{n-1}^{k+1}) \right),\\
w_i^{k+1} &=   w_i^k + \theta \left(x_{i+1}^{k+1}-x_{i}^{k+1}\right), \qquad \forall i\in \llbracket 1, n-1\rrbracket,
\end{aligned}\right.
\end{equation}
which allows solving problem~\eqref{eq:Prob} when $m=n-1$. Evaluating exactly one cocoercive operator inside each resolvent in~\eqref{eq:ringFB} serves to cover different settings, as the sum of cocoercive operators is itself cocoercive (see~\cite[Remark~2]{ring-networks}). Observe that, in contrast to the generalized forward-backward, algorithm \eqref{eq:ringFB} has minimal lifting and, now, it recovers the Davis--Yin algorithm as a special case.

Apart from having different lifting, algorithms \eqref{eq:genFB} and \eqref{eq:ringFB} exhibit contrasting structures of interdependence between the resolvent variables. In algorithm \eqref{eq:genFB}, variables $x_1^k, \ldots, x_n^k$ can be independently updated since none of them relies on the others, enabling thus a parallel implementation. In contrast, updating each $x_2^k,\ldots,x_{n-1}^k$ in algorithm \eqref{eq:ringFB} requires the preceding one, whereas $x_n^{k}$ depends on both $x_{n-1}^{k}$ and $x_1^k$. Consequently, the latter scheme is conducive to a decentralized implementation on a ring network topology. Subsequent developments have given rise to other schemes with different structures of interdependence between their variables (see, for instance, \cite{degenerate-ppp,condat2023proximal,MBG24,tam2023frugal}).

In the recent work~\cite{graph-drs}, the authors provide a unifying framework for systematically constructing frugal splitting algorithms with minimal lifting for finding zeros in the sum of $n$ maximally monotone operators (i.e., problem~\eqref{eq:Prob} with $m=0$). Their methodology involves reformulating the original monotone inclusion into an equivalent one which is described by a single operator constructed in the  larger space $\Hi^{2n-1}$. The relationships between the governing and resolvent variables are modeled through a connected directed graph and a subgraph, and the single operator is constructed in such a way that it integrates this information. The resulting monotone inclusion is addressed by using the \emph{degenerate preconditioned proximal point algorithm} of~\cite{degenerate-ppp}, where the preconditioner is defined through the subgraph.

In this work, we generalize and combine the methodologies presented in \cite{graph-drs} and \cite{degenerate-ppp} to also allow cocoercive operators to be integrated into the iterative process. This is done by incorporating an additional subgraph to model the action of the cocoercive operators. Our framework yields a novel family of forward-backward-type algorithms for solving~\eqref{eq:Prob}, accommodating different interdependence structures. In particular, it covers~\eqref{eq:ringFB} and other recently developed algorithms as special cases, and also allows to derive a promising novel forward-backward algorithm with full connectivity based on the complete graph.

The structure of the paper is as follows. We begin in Section~\ref{sec:prelim} by introducing the notation and main concepts. In Section~\ref{sec:operators}, we state the graph settings that give rise to our family of forward-backward algorithms, we construct the operators to which the preconditioned proximal point algorithm is applied, and analyze the main properties of these operators. In Section~\ref{sec:method}, we derive our main algorithm, prove its convergence and study some of its particular instances. Section~\ref{sec:experiments} is devoted to a numerical experiment in which we test how the graphs defining the algorithm affect the performance. We finish with some conclusions in Section~\ref{sec:conclusion}.

\section{Preliminaries}\label{sec:prelim}

Throughout this work, $\Hi$ is a real Hilbert space with inner product $\langle\cdot,\cdot\rangle$ and associated norm~$\norm{\cdot}$. We denote strong convergence of sequences by $\to$ and use $\rightharpoonup$ for weak convergence. Vectors in product spaces are marked with bold, e.g., $\mathbf{x}=(x_1,\ldots,x_n)\in\Hi^n$. If $(\Hi_i,\langle\cdot,\cdot\rangle_i)$ are Hilbert spaces for $i=1,\ldots,n$ and $\mathbf{x},\mathbf{y}\in\bigtimes_{i=1}^n \Hi_i$, then the operation $\langle \mathbf{x},\mathbf{y}\rangle:=\sum_{i=1}^n\langle x_i,y_i\rangle_i$ defines an inner product in $\bigtimes_{i=1}^n \Hi_i$.

A \emph{set-valued operator}, denoted by $A:\Hi\rightrightarrows \Hi$, is a map $A:\Hi\to 2^\Hi$, where $2^\Hi$ is the power set of $\Hi$. That is, for all $x\in \Hi$, $A(x)\subseteq \Hi$. On the other hand, if $B$ is an operator such that $B(x)$ is a singleton for all $x\in \Hi$, then $B$ is said the be a \emph{single-valued operator}, which is denoted by $B:\Hi\to \Hi$ and, by an abuse of notation, we will write $B(x)=y$ instead of $B(x)=\{y\}$.

Given a set-valued operator $A:\Hi\rightrightarrows \Hi$, the \emph{domain}, the \emph{range}, the \emph{graph}, the \emph{fixed points} and the \emph{zeros} of $A$ are, respectively,
\begin{align*}
\operatorname{dom}A&:=\{x\in \Hi: A(x)\neq\varnothing\}, & \operatorname{ran}A&:=\{u\in \Hi: u\in A(x)\text{ for some }x\in \Hi\},\\
\operatorname{gra}A&:=\{(x,u)\in \Hi\times \Hi: u\in A(x)\},& \operatorname{fix}A&:=\{x\in \Hi: x\in A(x)\},\\
\operatorname{zer}A&:=\{x\in \Hi: 0\in A(x)\}.
\end{align*}
The \emph{inverse} is the set-valued operator $A^{-1}:\Hi\rightrightarrows \Hi$ such that $x\in A^{-1}(u)\Leftrightarrow u\in A(x)$.

\begin{definition}\label{def: monotone}
    We say that a set-valued operator $A:\Hi\rightrightarrows \Hi$ is \emph{monotone} if
    $$\langle x-y,u-v\rangle\geq0\quad \forall (x,u),(y,v)\in \operatorname{gra}A.$$
    Further, $A$ is \emph{maximally monotone} if for all $A^\prime:\Hi\rightrightarrows \Hi$ monotone, $\operatorname{gra}A\subseteq\operatorname{gra}A^\prime$ implies $A=A^\prime$.
\end{definition}

\begin{proposition}[{\cite[Corollary 20.28]{bauschke}}]\label{prop: maximally monotone continuous}
    Let $A:\Hi\to \Hi$ be monotone and continuous. Then $A$ is maximally monotone.
\end{proposition}

\begin{proposition}[{\cite[Corollary 25.5]{bauschke}}]\label{prop: sum of maximally monotone}
    Let $A_1,A_2:\Hi\rightrightarrows \Hi$ be maximally monotone. If $\operatorname{dom}A_2=\Hi$, then $A_1+A_2$ is maximally monotone.
\end{proposition}

In the context of splitting algorithms, the resolvent operator is central. It is defined as follows.

\begin{definition}\label{def: resolvent}
    Let $A:\Hi\rightrightarrows \Hi$ be a set-valued operator. The \emph{resolvent} of $A$ is
    \[J_A:=(\Id_\Hi+A)^{-1}.\]
\end{definition}

\begin{lemma}[{\cite{minty}}]\label{lem: single-valued full domain resolvent}
    Let $A:\Hi\rightrightarrows \Hi$ be monotone. Then:
    \begin{enumerate}[label=(\roman*)]
        \item $J_A$ is single-valued;
        \item $A$ is maximally monotone if and only if $\operatorname{dom} J_A=\Hi$.
    \end{enumerate}
\end{lemma}

Let us now turn our attention to single-valued operators. Let $B:\Hi\to\Hi$  be a linear operator. We say that $B$ is \emph{bounded} if there exists some $\kappa>0$ such that  $\norm{B(x)}\leq \kappa\norm{x}$, for all $x\in \Hi$. We denote by $B^\ast$ the \emph{adjoint} of $B$, i.e., the linear operator $B^\ast:\Hi\to \Hi$ such that $\langle B(x),y\rangle=\langle x,B^\ast (y)\rangle$ for all $x,y\in \Hi$.

\begin{definition}\label{def: cocoercive lipschitz}
    Let $B:\Hi\to \Hi$ be a single-valued operator and let $\beta,L>0$.
    \begin{enumerate}[label=(\roman*)]
        \item $B$ is $\beta$\emph{-cocoercive} if $\langle B(x)-B(y),x-y\rangle\geq\beta\norm{B(x)-B(y)}^2$ for all $x,y\in \Hi$.
        \item $B$ is $L$\emph{-Lipschitz continuous} if $\norm{B(x)-B(y)}\leq L\norm{x-y}$ for all $x,y\in \Hi$.
    \end{enumerate}
\end{definition}
    We simply say that an operator is cocoercive or Lipschitz continuous when it is not necessary to specify the constants.%
Clearly, every $\beta$-cocoercive operator is $\frac{1}{\beta}$-Lipschitz continuous by the Cauchy--Schwarz inequality, although $\frac{1}{\beta}$ might not be the best Lipschitz constant (see {\cite[Remark 4.15]{bauschke}}).

\begin{definition}\label{def: self-adjoint orthogonal PSD}
    Let $B:\Hi\to \Hi$ be a linear operator.
    \begin{enumerate}[label=(\roman*)]
        \item $B$ is \emph{self-adjoint} if $B=B^\ast$.
        \item $B$ is \emph{orthogonal} if $B$ is an isomorphism and $B^{-1}=B^\ast$.
        \item $B$ is \emph{positive semidefinite} if $\langle B(x),x\rangle\geq0$ for all $x\in \Hi$.
    \end{enumerate}
\end{definition}

Finally, we will make use of the following construction in our developments.

\begin{definition}\label{def: parallel composition}
    Let $A:\Hi\rightrightarrows \Hi$ and let $L:\Hi\to \Hi$ be linear. The \emph{parallel composition} of $A$ and $L$ is the set-valued operator
    \[L\rhd A:=(LA^{-1}L^\ast)^{-1}.\]
\end{definition}

\subsection{Preconditioned proximal point algorithms}\label{sec:PPPA}

A wide family of optimization methods are designed to solve inclusion problems of set-valued operators. That is to say, given a set-valued operator $A:\Hi\rightrightarrows \Hi$ , we are interested in the following problem:
\begin{equation}\label{eq: inclusion problem}
    \text{find }x\in \Hi\text{ such that }0\in A(x).
\end{equation}
One of the most popular algorithms to tackle this problem is the \emph{proximal point algorithm}, which is based on transforming inclusion~\eqref{eq: inclusion problem} into a fixed-point problem as follows. Given any $\gamma>0$, it holds
\begin{equation}\label{eq: inclusion to fixed point}
    0\in A(x)\Leftrightarrow x\in \gamma A(x)+x=(\gamma A+\operatorname{Id}_\Hi)(x)\Leftrightarrow x\in J_{\gamma A}(x),
\end{equation}
where $\operatorname{Id}_\Hi$ denotes the identity mapping on $\Hi$. To construct a uniquely determined fixed point iteration from~\eqref{eq: inclusion to fixed point}, the resolvent must be a single-valued operator with full domain which, by Lemma~\ref{lem: single-valued full domain resolvent}, is the same as requiring $A$ to be maximally monotone. In this way, the proximal point algorithm is defined by the iterative scheme
\begin{equation}\label{eq: proximal point}
    x^{k+1}=J_{\gamma A}(x^k),\quad k=0,1,2,\ldots,
\end{equation}
where the parameter $\gamma>0$ is referred to as the \emph{stepsize}.

If $A$ is simple enough to have a computable resolvent, this method suffices to obtain a good approximation to the solution. Nevertheless, for a general set-valued operator, computing the resolvent might be as difficult as solving the original inclusion problem~\eqref{eq: inclusion problem}. To facilitate the computation of the resolvent, we could replace the identity operator $\operatorname{Id}_\Hi$ in \eqref{eq: inclusion to fixed point} by some other linear and bounded operator $M:\Hi\to \Hi$, giving rise to the following fixed-point problem:
\begin{equation}
    0\in A(x)\Leftrightarrow Mx\in A(x)+Mx=(A+M)(x)\Leftrightarrow x\in(A+M)^{-1}(Mx).
\end{equation}
Notice that by doing some simple algebraic manipulations we get that $(A+M)^{-1}M=J_{M^{-1}A}$, see~\cite[Eq.~(2.3)]{degenerate-ppp} for details. The method obtained in this way is called the  \emph{preconditioned proximal point algorithm} and is given by
\begin{equation}\label{eq: preconditioned proximal point}
    x^{k+1}=J_{M^{-1}A}(x^k),\quad k=0,1,2,\ldots.
\end{equation}
To make sense of of this equation, we need to ensure that the resolvent $J_{M^{-1}A}$ has full domain and is single-valued, hence motivating the following definition.
\begin{definition}\label{def: admissible preconditioner}
    Let $A:\Hi\rightrightarrows \Hi$ be a set-valued operator and $M:\Hi\to \Hi$ be a linear, bounded, self-adjoint and positive semidefinite operator. We say that $M$ is an \emph{admissible preconditioner} for $A$ if
    \[J_{M^{-1}A}\text{ is single-valued and has full domain.}\]
\end{definition}

It is convenient to work with a generalized form of \eqref{eq: preconditioned proximal point} where certain parameters are allowed. Specifically, given a sequence of relaxation parameters $\{\theta_k\}_{k=0}^{\infty}$ such that
\begin{equation}\label{def: relaxation paramter}
\theta_k\in{]0,2]} \text{ for all } k\in\mathbb{N} \quad\text{and}\quad \sum_{k\in\mathbb{N}}\theta_k(2-\theta_k)=+\infty,
\end{equation}
a generalized form of \eqref{eq: preconditioned proximal point}, which is called the \emph{relaxed preconditioned proximal point algorithm}, is given by
\begin{equation}\label{eq: relaxed preconditioned proximal point}
    x^{k+1}=x^k+\theta_k(J_{M^{-1}A}(x^k)-x^k),\quad k=0,1,2,\ldots.
\end{equation}

Since the admissible preconditioner $M$ is self-adjoint and positive semidefinite, by~\cite[Proposition 2.3]{degenerate-ppp}, it can be split as $M=CC^\ast$ for some injective operator $C$. When $\operatorname{ran}M$ is closed, this factorization is called an \emph{onto decomposition} of $M$. Such factorization is unique up to orthogonal transformations, see~\cite[Proposition~2.2]{graph-drs}. As we subsequently detail, this allows  rewriting~\eqref{eq: relaxed preconditioned proximal point} in terms of the resolvent of the parallel composition $C^\ast\rhd A$.

\begin{lemma}[{\cite[Theorem 2.13]{degenerate-ppp}}]\label{lem: parallel resolvent}
    Let $A:\Hi\rightrightarrows \Hi$, let $M$ be an admissible preconditioner for $A$ and let $M=CC^\ast$ be an onto decomposition. Then $C^\ast \rhd A$ is maximally monotone and
    \[J_{C^\ast\rhd A}=C^\ast(M+A)^{-1}C.\]
\end{lemma}
Thanks to this lemma, we can bring \eqref{eq: relaxed preconditioned proximal point} back and modify it to rewrite it using the resolvent of the parallel composition $J_{C^\ast \rhd A}$. To do this, recall that $J_{M^{-1}A}=(M+A)^{-1}M=(M+A)^{-1}CC^\ast$. Introducing the new variable $y^k:=C^\ast x^k$, we can write \eqref{eq: relaxed preconditioned proximal point} as
\begin{equation}\label{eq: rPPP yk}
    x^{k+1}=x^k+\theta_k((M+A)^{-1}(Cy^k)-x^k).
\end{equation}
Now, operating by $C^\ast$ on both sides of equation~\eqref{eq: rPPP yk}, we obtain
\begin{equation}\label{eq: rPPP C*x}
    C^\ast x^{k+1}=C^\ast x^k+\theta_k(C^\ast(M+A)^{-1}(Cy^k)+C^\ast x^k).
\end{equation}
Using Lemma~\ref{lem: parallel resolvent} and the definition of $y^k$ in equation~\eqref{eq: rPPP C*x}, we obtain the so-called \emph{reduced preconditioned proximal point algorithm}, which is given by
\begin{equation}\tag{RPPA}\label{eq: reduced preconditioned proximal point}
    y^{k+1}=y^k+\theta_k(J_{C^\ast\rhd A}(y^k)-y^k),\quad k=0,1,2,\ldots.
\end{equation}

\begin{theorem}[{\cite[Theorem 2.5]{graph-drs}}]\label{th: RPPP convergence}
    Let $A:\Hi\rightrightarrows \Hi$ be maximally monotone and suppose that $\zer A\neq\varnothing$. Let $M$ be and admissible preconditioner for $A$ and let $M=CC^\ast$ be an onto decomposition with $C:\Hi^\prime\to \Hi$. Pick any $y^0\in \Hi^\prime$ and iteratively define the sequence $\{y^k\}_{k=0}^\infty$ by~\eqref{eq: reduced preconditioned proximal point}. Then, the following assertions hold.
    \begin{enumerate}[label=(\roman*)]
        \item There exists some $y^\ast\in \Hi^\prime$ such that $y^k\wto y^\ast$ and $u^\ast:=(M+A)^{-1}(Cy^\ast)\in\zer A$.
        \item If $(M+A)^{-1}$ is Lipschitz, then $(M+A)^{-1}(Cy^k)\wto u^\ast$.
    \end{enumerate}
\end{theorem}

\subsection{Graph theory}

We say that $G=(\mathcal{N},\mathcal{E})$ is a \emph{(directed) graph} if $\mathcal{N}$ is a finite set and $\mathcal{E}\subseteq \mathcal{N}\times \mathcal{N}$. The elements of $\mathcal{N}$ are known as \emph{nodes} and the elements of $\mathcal{E}$ are called \emph{edges}. The \emph{order} of the graph is the number of nodes $|\mathcal{N}|$. Notice also that, since $\mathcal{N}$ is finite, we can relabel its nodes as $\mathcal{N}=\{1,\ldots,n\}$ for some $n\geq 1$. A graph can be depicted as dots representing the nodes and arrows connecting one node to another, representing the edges.

\begin{definition}\label{def: degree matrix}
    Let $G=(\mathcal{N},\mathcal{E})$ be a graph of order $n$.
    \begin{enumerate}[label=(\roman*)]
        \item A \emph{subgraph} of $G$ is a graph $G^\prime=(\mathcal{N}^\prime,\mathcal{E}^\prime)$ such that $\mathcal{N}\subseteq \mathcal{N}^\prime$ and $\mathcal{E}\subseteq\mathcal{E}^\prime$. By abuse of notation, we will write $G^\prime\subseteq G$. We say that $G^\prime\subseteq G$ is a \emph{spanning subgraph} if $\mathcal{N}^\prime=\mathcal{N}$.

        \item A node $j\in\mathcal{N}$ is said to be \emph{adjacent} to a node $i\in\mathcal{N}$ if $(i,j)\in\mathcal{E}$. The \emph{adjacency matrix} of $G$ is the matrix $\operatorname{Adj}(G)\in\mathbb{R}^{n\times n}$ defined componentwise as
        \[\operatorname{Adj}(G)_{ij}:=\begin{cases}
            1 & \text{if }(i,j)\in\mathcal{E},\\
            0 & \text{otherwise}.
        \end{cases}\]

        \item The \emph{in-degree} and \emph{out-degree} of a node $i\in\mathcal{N}$ are
        defined as $d_i^{\rm in}:=|\{j\in\mathcal{N}:(j,i)\in\mathcal{E}\}|$ and
        $d_i^{\rm out}:=|\{j\in\mathcal{N}:(i,j)\in\mathcal{E}\}|$, respectively.
        The \emph{degree} is the sum of both numbers and is denoted by $d_i:=d_i^{\rm in}+d_i^{\rm out}$.
        The \emph{degree matrix} of $G$ is the diagonal matrix
        \[\operatorname{Deg}(G):=\operatorname{diag}(d_1,\ldots,d_n)\in\mathbb{R}^{n\times n}.\]

        \item A \emph{(weak) path} in $G$ is a finite sequence of all distinct nodes $(v_1,\ldots,v_r)$ with $r\geq2$ such that $(v_k,v_{k+1})\in\mathcal{E}$ or $(v_{k+1},v_k)\in\mathcal{E}$ for all $k=1,\ldots,r-1$. In this context, we say that $v_1$ and $v_r$ are the \emph{endpoints} of the path.

        \item Two distinct nodes $i,j\in\mathcal{N}$ are \emph{connected} if there exists a path with $i$ and $j$ as endpoints. A graph $G$ is \emph{connected} if every pair of nodes are connected.
    \end{enumerate}
\end{definition}

Since $\mathcal{E}\subseteq \mathcal{N}\times\mathcal{N}$, the following special situations may occur: (i) if there is a node $i\in\mathcal{N}$ such that $(i,i)\in\mathcal{E}$, then it forms a \emph{loop}; (ii) if there are two nodes $\{i,j\}\subseteq\mathcal{N}$ such that $(i,j),(j,i)\in\mathcal{E}$, then they form a \emph{2-cycle}.

\begin{definition}\label{def: oriented tree}
    Let $G=(\mathcal{N},\mathcal{E})$ be a graph. Then $G$ is:
    \begin{enumerate}[label=(\roman*)]
        \item \emph{oriented} if it contains no loops and no 2-cycles;
        \item a \emph{tree} if it is connected and $|\mathcal{E}|=|\mathcal{N}|-1$.
    \end{enumerate}
\end{definition}

One can easily check that every tree is oriented. On other other hand, not every oriented graph is connected.

Notice that, since the number of nodes of a graph is finite, so is the number of edges. Hence, we can also index them with natural numbers as $\mathcal{E}=\{1,\ldots,E\}$, which is useful for defining the following matrices.

\begin{definition}\label{def: incidence matrix}
    Let $G=(\mathcal{N},\mathcal{E})$ be an oriented graph of order $n$ and let $E:=|\mathcal{E}|$.
    \begin{enumerate}[label=(\roman*)]
    \item The \emph{incidence matrix} of $G$ is the matrix $\operatorname{Inc}(G)\in\mathbb{R}^{n\times E}$ defined componentwise as
    \[\operatorname{Inc}(G)_{ie}:=\begin{cases}
        1 & \text{if the edge }e\text{ leaves the node }i,\\
        -1 & \text{if the edge }e\text{ enters the node }i,\\
        0 & \text{otherwise}.
    \end{cases}\]
    \item The \emph{Laplacian matrix} is the symmetric and positive semidefinite matrix defined as  $\operatorname{Lap}(G):=\operatorname{Inc}(G)\operatorname{Inc}(G)^\ast\in\mathbb{R}^{n\times n}$, which can be described componentwise as
    \[\operatorname{Lap}(G)_{ij}:=\begin{cases}
        d_i & \text{if }i=j,\\
        -1 & \text{if }(i,j)\in\mathcal{E}\text{ or }(j,i)\in\mathcal{E},\\
        0 & \text{otherwise.}
    \end{cases}\]
    Thus, $\operatorname{Lap}(G)=\operatorname{Deg}(G)-\operatorname{Adj}(G)-\operatorname{Adj}(G)^\ast$.
    \end{enumerate}
\end{definition}

The next results collect the main properties of the incidence and Laplacian matrix that we employ.

\begin{lemma}[{\cite[Theorems 8.3.1 and 13.1.1]{godsil}}]\label{lem: rank incidence}
    If $G$ is a connected oriented graph of order $n$, then $\rk(\operatorname{Inc}(G))=n-1$. Hence, $\rk(\operatorname{Lap}(G))=n-1$ and $\ker(\operatorname{Lap}(G))=\operatorname{span}\{\mathbf{1}\}$, where $\mathbf{1}=(1,\ldots,1)\in\R^n$.
\end{lemma}

\begin{proposition}\label{prop: Zdecom}
Let $G$ be a connected oriented graph of order $n$. Then, there exists a matrix $Z\in\mathbb{R}^{n\times (n-1)}$ such that
\begin{equation}\label{eq:Zdecom}
\operatorname{Lap}(G)=ZZ^\ast.
\end{equation}
Consequently, $\rk Z = n-1$ and $\ker Z^\ast=\operatorname{span}\{\mathbf{1}\}$. In particular, if $\operatorname{Lap}(G)=Q\Lambda Q^\ast$ is a spectral decomposition of $\operatorname{Lap}(G)$, where $\Lambda=\operatorname{diag}(\lambda_1,\ldots,\lambda_{n-1},0)$ is the diagonal matrix of eigenvalues with $\lambda_1\geq\ldots\geq\lambda_{n-1}\geq 0$ and the columns of $Q=[v_1~v_2~\cdots~v_n]\in\R^{n\times n}$ correspond to eigenvectors, an onto decomposition of $\operatorname{Lap}(G)$ is given by
$$Z:=[v_1~v_2~\cdots~v_{n-1}]\operatorname{diag}(\sqrt{\lambda_1},\ldots,\sqrt{\lambda_{n-1}}).$$
\end{proposition}
\begin{proof}
This is a direct consequence of the definition of $\operatorname{Lap}(G)$ and Lemma~\ref{lem: rank incidence}.
\end{proof}

\begin{remark}\label{rem: tree incidence}
    The previous proposition gives us a constructive approach to produce an onto decomposition of the Laplacian matrix by means of its eigendecomposition. Nonetheless, if $G$ is a tree, we can directly take the incidence matrix of $G$ as $Z$ in \eqref{eq:Zdecom}. Indeed, by definition of the Laplacian, we have that $\operatorname{Lap}(G)=\operatorname{Inc}(G)\operatorname{Inc}(G)^\ast$. Since $G$ is a tree, it has $n-1$ edges and then  $\operatorname{Inc}(G)\in\mathbb{R}^{n\times (n-1)}$.
\end{remark}

\section{Algorithm framework}\label{sec:operators}

In this section, we present some suitable settings for the underlying graphs that give rise to a family of frugal forward-backward splitting algorithms for solving problem \eqref{eq:Prob} when $m=n-1$. As in~\cite{graph-drs}, the algorithms are devised as an application of~\eqref{eq: reduced preconditioned proximal point} to some ad hoc operators acting in a product Hilbert space. More precisely, we define a set-valued operator $\mathcal{A}$ based on the maximally monotone operators, a single-valued operator $\mathcal{B}$ based on the cocoercive ones and an admissible preconditioner $\mathcal{M}$ for the maximally monotone operator $\mathcal{A}+\mathcal{B}$. Further, this is done in such a way that it is straightforward to derive solutions to the original inclusion problem from the set of zeros of $\mathcal{A}+\mathcal{B}$.

\subsection{Graph settings}

As mentioned above, the operators $\mathcal{A}$, $\mathcal{B}$ and $\mathcal{M}$ are defined based on the underlying graph structure. Hence, we must impose some properties to the graph in order to depict which variables are needed to evaluate each resolvent variable.

\begin{definition}\label{def: algorithmic graph}
    We say that $G=(\mathcal{N},\mathcal{E})$ is an \emph{algorithmic graph} if
    \begin{enumerate}[label=(\roman*)]
        \item $\mathcal{N}=\{1,\ldots,n\}$ with $n\geq2$,
        \item $(i,j)\in\mathcal{E}\Rightarrow i<j$, and
        \item $G$ is connected.
    \end{enumerate}
\end{definition}

Observe that, by definition, every algorithmic graph is oriented. Thus, the adjacency, incidence and Laplacian matrices for these graphs are well-defined. Let us present some examples of algorithmic graphs.

\begin{example}[Sequential graph]\label{ex: sequential graph}
    For every $n\geq2$, there is a unique algorithmic graph of order $n$ which is a path (and hence, a tree), and we refer to it as \emph{sequential}. The degrees of the nodes are $d_1=d_n=1$, while $d_i=2$ for all $i=2,\ldots,n-1$. It is represented in Figure~\ref{fig: sequential graph}.
\end{example}

\begin{example}[Ring graph]\label{ex: ring graph}
    For every $n\geq2$, the \emph{ring} graph is a cycle of order $n$ whose edges are always forwardly directed. The degrees of the nodes are $d_i=2$ for all $i\in\mathcal{N}$. It is depicted in Figure~\ref{fig: ring}.
\end{example}

\begin{example}[Parallel graph]\label{ex: parallel graph}
    The graphs with edges given by $\mathcal{E}_u=\{(1,j): j=2,\ldots,n\}$ or $\mathcal{E}_d=\{(i,n): i=1,\ldots,n-1\}$ are algorithmic trees. Both have the same underlying graph structure, namely, a star graph. In this setting, there is a node with degree $n-1$, whereas the rest have degree 1. We refer to these types of algorithmic graphs as \emph{parallel}. Specifically, the graph with edges $\mathcal{E}_u$ is known as \emph{parallel up}, see Figure~\ref{fig: parallel up}, and the one with $\mathcal{E}_d$ is called \emph{parallel down}, see Figure~\ref{fig: parallel down}.
\end{example}

\begin{example}[Complete graph]\label{ex: complete graph}
    The graph of order $n$ given by $\mathcal{E}=\{(i,j): i<j\}$ is an algorithmic graph that is called \emph{complete}. The degree of every node is $d_i=n-1$ for all $i\in\mathcal{N}$. It is illustrate it in Figure~\ref{fig: complete graph}.
\end{example}

For $n=2$, there is a unique algorithmic graph whose nodes have degree $1$. For $n=3$, every algorithmic graph is either one of the examples presented above: complete, sequential, parallel up or parallel down (the ring and complete graphs coincide).

\begin{example}[Union of graphs]\label{ex: biparallel graph}
    Given two algorithmic graphs $G_1=(\mathcal{N},\mathcal{E}_1)$ and $G_2=(\mathcal{N},\mathcal{E}_2)$ of order $n$, we can construct an algorithmic graph by just taking the union $G=G_1\cup G_2$. Formally, it is a graph with the same set of nodes $\mathcal{N}$ and $\mathcal{E}:=\mathcal{E}_1\cup\mathcal{E}_2$. For example, the union of a parallel up and a parallel down of order $n$, which we call \emph{biparallel}, has the structure depicted in Figure~\ref{fig: biparallel graph}.
\end{example}

\newcommand\ancho{.48}
\newcommand\escala{.9}
\newcommand\espacio{\hspace{0.4cm}}
\vspace{-0.5cm}
\begin{figure}[ht!]
	\setlength{\abovecaptionskip}{5pt}
	\setlength{\belowcaptionskip}{0pt}
    \centering
    \begin{subfigure}{\ancho\textwidth}
    \begin{center}
    \adjustbox{scale=\escala,center}{%
        \begin{tikzcd}[cells={nodes={draw=black, circle}}]
        1 \arrow[r] & 2 \arrow[r] & 3 \arrow[r] & |[draw=none]|\cdots\vphantom{1} \arrow[r] & n\vphantom{1}
        \end{tikzcd}
        }
    \end{center}
    \caption{Sequential}\label{fig: sequential graph}
    \end{subfigure}
    \espacio
    \begin{subfigure}{\ancho\textwidth}\centering
    \adjustbox{scale=\escala,center}{%
        \begin{tikzcd}[cells={nodes={draw=black, circle}}]
        1 \arrow[rrrr, bend left=25] \arrow[r] & 2 \arrow[r] & 3 \arrow[r]  & |[draw=none]|\cdots\vphantom{1} \arrow[r] & n\vphantom{1}
        \end{tikzcd}
        }
         \vspace{0.27cm}
        	\caption{Ring}\label{fig: ring}
    \end{subfigure}\\
    \begin{subfigure}{\ancho\textwidth}\centering
    \adjustbox{scale=\escala,center}{%
            \begin{tikzcd}[cells={nodes={draw=black, circle}}]
                1 \arrow[rrrr, bend left=30] \arrow[rrr, bend left=30] \arrow[rr, bend left] \arrow[r]  & 2 & 3 & |[draw=none]|\cdots & n
                \end{tikzcd}
                }
                \vspace{-0.12cm}
            \caption{Parallel up}\label{fig: parallel up}
    \end{subfigure}
    \espacio
    \begin{subfigure}{\ancho\textwidth}\centering
    \adjustbox{scale=\escala,center}{%
            \begin{tikzcd}[cells={nodes={draw=black, circle}}]
                1 \arrow[rrrr, bend left=35] & 2 \arrow[rrr, bend left=25] & 3 \arrow[rr, bend left=15] & |[draw=none]|\cdots\vphantom{1} \arrow[r] & n\vphantom{1}
                \end{tikzcd}
                }
            \caption{Parallel down}\label{fig: parallel down}
    \end{subfigure}\\
    \begin{subfigure}{\ancho\textwidth}
    \begin{center}
    \adjustbox{scale=\escala,center}{%
        \begin{tikzcd}[cells={nodes={draw=black, circle}}]
            1 \arrow[rrr, bend left=49] \arrow[r] \arrow[rr, bend left] \arrow[rrrr, bend left=49] & 2 \arrow[r] \arrow[rr, bend right] \arrow[rrr, bend right] & 3 \arrow[r] \arrow[rr, bend left] & |[draw=none]|\cdots\vphantom{1} \arrow[r] & n\vphantom{1}
            \end{tikzcd}
            }
            \vspace{-0.13cm}
    \end{center}
    \caption{Complete}\label{fig: complete graph}
    \end{subfigure}
    \espacio
    \begin{subfigure}{\ancho\textwidth}\centering
        \adjustbox{scale=\escala,center}{%
    \begin{tikzcd}[cells={nodes={draw=black, circle}}]
    1 \arrow[rrrr, bend left=40] \arrow[rrr, bend left=35] \arrow[rr, bend left] \arrow[r, bend left] & 2 \arrow[rrr, bend right=40] & 3 \arrow[rr, bend right=30] & |[draw=none]|\cdots\vphantom{1} \arrow[r, bend right=20] & n
    \end{tikzcd}
    }
            \caption{Biparallel}\label{fig: biparallel graph}
    \end{subfigure}
    \caption{Examples of algorithmic graphs}
    \end{figure}
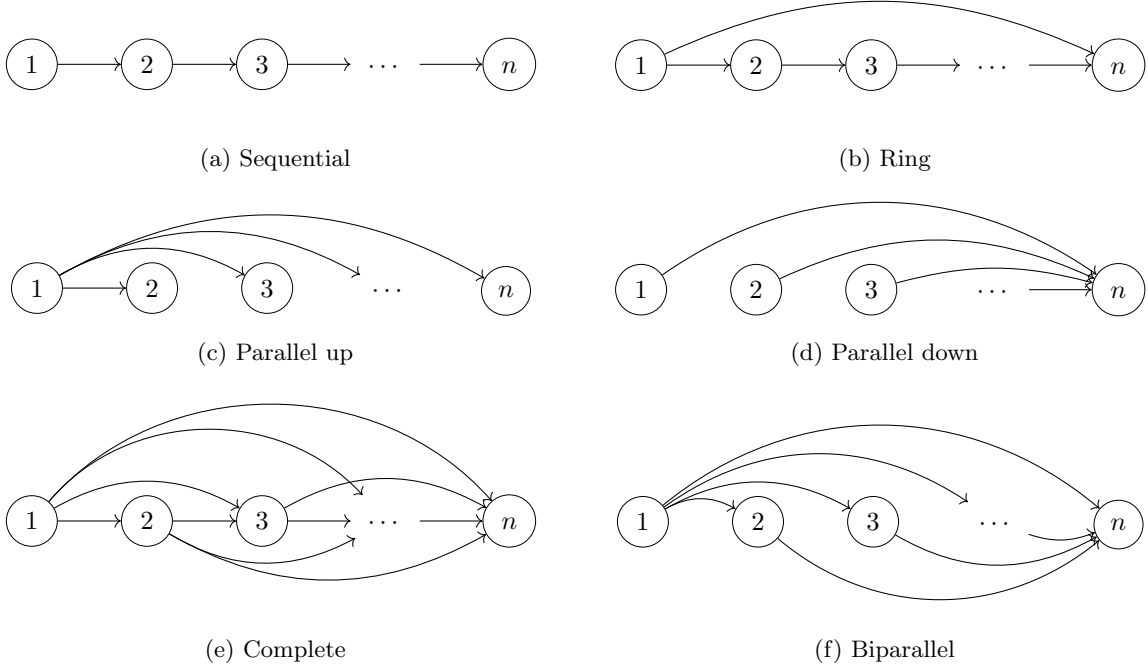

We present next our standing hypotheses on the graphs that define our family of algorithms for solving~\eqref{eq:Prob} with $m=n-1$.

\begin{assumption} \label{assgraphs} Let $A_1,\ldots,A_n:\Hi\tto\Hi$ be maximally monotone and let $B_1$,$\ldots$,$B_{n-1}:\Hi\to\Hi$ be $\beta$-cocoercive operators, with $\operatorname{zer}\left(\sum_{i=1}^nA_i + \sum_{i=1}^{n-1}B_i\right)\neq\varnothing$. We assume that there is a triple of graphs $(G, G', G'')$ of order $n\geq 2$ verifying the following conditions:

\begin{enumerate}[leftmargin=1.8cm,label=(AG\arabic*)]
    \item\label{axi: alg graph} $G=(\mathcal{N},\mathcal{E})$ is an algorithmic graph.
    \item\label{axi: alg subgraph} $G^\prime=(\mathcal{N},\mathcal{E}^\prime)\subseteq G$ is a connected spanning subgraph of $G$.
    \item\label{axi: alg subtree} $G^{\prime\prime}=(\mathcal{N},\mathcal{E}^{\prime\prime})\subseteq G$ is a spanning subgraph of $G$ such that $d_i^{\rm in}=1$ for all $i\geq 2$. Hence, for each $i\geq 2$ there is a unique $\pre(i)\in\mathcal{N}$ such that $(\pre(i),i)\in\mathcal{E}^{\prime\prime}$.
\end{enumerate}
\end{assumption}

\begin{remark}[On the role of the graphs] \label{rem:graphs}
Let us describe how each of the graphs stated in Assumption~\ref{assgraphs} plays a role in defining our iterative algorithm, presented in Section~\ref{sec:method}:
\begin{enumerate}[label=(\roman*)]
\item The algorithmic graph $G$ depicts which resolvent variables are used to update each other. Specifically, if $(i,j)\in\mathcal{E}$, then to update $x_j^{k+1}$ using the (parametrized) resolvent of $A_j$ it is necessary to evaluate the variable $x_i^{k+1}$.

\item The first subgraph $G^\prime$ is employed to gather different algorithms  within the same family. It determines how the governing and resolvent variables of the scheme interact at each iteration. Specifically, given an onto decomposition $Z\in\mathbb{R}^{n\times (n-1)}$ of $\operatorname{Lap}(G^\prime)$ (see Proposition~\ref{prop: Zdecom}), updating the resolvent variable $x_j^k$ the algorithm will use those governing variables $w_e^k$ for which $Z_{je}\neq 0$. When $G^\prime$ is a tree, this graph explicitly determines through the incidence matrix which governing variables are evaluated to update each resolvent variable (recall Remark~\ref{rem: tree incidence}). That is, updating the resolvent variable $x_j^{k+1}$ will require a combination of those governing variables $w_{e^{\rm in}}^k$ and $w_{e^{\rm out}}^k$ for which $e^{\rm in}=(i^{\rm in},j)\in\mathcal{E}^\prime$ and $e^{\rm out}=(j,i^{\rm out})\in\mathcal{E}^\prime$, with $i^{\rm in},i^{\rm out}\in\mathcal{N}$. In addition, the subgraph $G^\prime$ also affects the inverse dependence relation between resolvent and governing variables; that is, the onto decomposition $Z$ determines how the resolvent variables are combined to update the governing sequence at the end of the current iteration, once all the resolvent variables have been updated.

\item The second subgraph $G^{\prime\prime}$ determines which resolvent variable is evaluated at each cocoercive operator. Namely, if $(i,j)\in\mathcal{E}^{\prime\prime}$, then to evaluate the resolvent of $A_j$ (which updates $x_j^{k+1}$), the algorithm computes the forward operation $B_j(x_i^{k+1})$. The additional assumption on the in-degrees of $G^{\prime\prime}$ restricts us to compute only one forward operation at each resolvent, so the resulting algorithm is frugal.
\end{enumerate}

\end{remark}
For further clarification, we illustrate our specific choice of graphs for a particular algorithm in the next example. It will be revisited in Example~\ref{ex: ring}, where the full expression of the method as a particular case of our algorithm will be deduced.

\begin{example} \label{ex:graph_ring}
Let us construct the specific graphs that permit to model algorithm~\eqref{eq:ringFB}.  To this end, we set $\mathcal{N}=\{1,2,\ldots,n\}$ and follow each item in Remark~\ref{rem:graphs}:
\begin{enumerate}[label=(\roman*)]
    \item Observe that updating each $x_2^{k+1},\ldots,x_{n}^{k+1}$  requires the preceding one. Additionally, $x_n^{k+1}$  also relies on $x_1^{k+1}$. Then, our algorithmic graph $G$ is the ring (see Example~\ref{ex: ring graph}).
    \item Now we look at the dependence between resolvent and governing variables. Besides $w_i^k$, also $w_{i-1}^k$ is employed to update each variable $x_i^{k+1}$ for $i=2,3,\ldots,n-1$, while only $w_1^k$ is used to compute $x_1^{k+1}$, and the only governing variable taken into account to update $x_n^{k+1}$ is $w_{n-1}^k$. Thus, the subgraph $G^\prime$ must be sequential (see Example~\ref{ex: sequential graph}).
    \item Finally, each resolvent $B_i$ is evaluated at $x_i^{k+1}$ to update $x_{i+1}^{k+1}$, for all $i=1,\ldots,n-1$. Hence, one needs to take $G^{\prime\prime}=G^\prime$.
\end{enumerate}

\end{example}

\begin{remark}[On the cocoercivity constant]\label{rem: coco constant}
Note that in Assumption~\ref{assgraphs} the cocoercivity constant $\beta$ is assumed to be the same for all cocoercive operators. Hence, if the operators $B_i$ are $\beta_i$-cocoercive with tight constants, the largest constant we can take is $\beta:=\min\{\beta_1,\ldots,\beta_{n-1}\}$. This means that if any operator has a small cocoercivity constant, then $\beta$ will be small. The value of $\beta$ affects the stepsizes allowed by our algorithm, which are bounded in $]0,4\beta[$, so a small value of $\beta$ entails small stepsizes. An alternative would be to set all cocoercive operators equal to zero but one, which is taken as $B:=\sum_{i=1}^{n-1}B_i$. By~\cite[Proposition~4.12]{bauschke}, the operator $B$ is $\widehat\beta$-cocoercive for $\widehat\beta:=(\sum_{i=1}^{n-1} \beta_i^{-1})^{-1}$. Since $\beta\leq\widehat\beta$ with strict inequality when not all $\beta_i$ are equal, the resulting algorithm would permit to choose a larger range of stepsizes. The price to pay is that this algorithm  no longer permits a \emph{distributed implementation} (see, e.g.,~\cite{ring-networks}), as one of the nodes will need to have access to all $B_i$ to evaluate the operator $B$. Therefore, our framework is flexible to cover both implementations.
\end{remark}

Before constructing the desired operators $\mathcal{M}$, $\mathcal{A}$ and $\mathcal{B}$, we end this subsection by defining two additional matrices related to the algorithmic graphs which will be useful for our subsequent analysis.

\begin{definition}\label{def: P Q}
    Let $G$ be an oriented graph. We define the following matrices:
    \begin{enumerate}[label=(\roman*)]
        \item $P(G):=\operatorname{Deg}(G)-2\operatorname{Adj}(G)^\ast$,
        \item $Q(G):=\operatorname{Adj}(G)-\operatorname{Adj}(G)^\ast$.
    \end{enumerate}
\end{definition}

\begin{remark}\label{rem: P laplacian}
    By the definition of $P(G)$, we clearly get that $\operatorname{Lap}(G)=\frac{P(G)+P(G)^\ast}{2}$. On the other hand, $Q(G)$ is skew-symmetric.
\end{remark}

\subsection{Construction of the operators}
The operators $\mathcal{A}$ and $\mathcal{M}$ are designed as in~\cite{graph-drs}, while the operator $\mathcal{B}$
extends the constructions of~\cite{degenerate-ppp} to appropriately include the cocoercive operators according to the graph structure.

\paragraph{The preconditioner \texorpdfstring{$\mathcal{M}$}{M}:}

Let us denote
$$\mathcal{L}^\prime:=\operatorname{Lap}(G^\prime)\otimes\operatorname{Id}_\Hi \quad\text{and}\quad \mathcal{Z}:=Z\otimes\operatorname{Id}_\Hi,$$
where $\otimes$ denotes the Kronecker product and $Z\in\mathbb{R}^{n\times (n-1)}$ is an onto decomposition of $\operatorname{Lap}(G^\prime)$ as in Proposition~\ref{prop: Zdecom}.  Then, it is straightforward to see that $\mathcal{L}^\prime=\mathcal{Z}\mathcal{Z}^\ast$. We define the preconditioner $\mathcal{M}:\Hi^{2n-1}\to \Hi^{2n-1}$ as the positive semidefinite linear operator
\begin{equation}\label{eq: M}
\mathcal{M}:=\begin{bmatrix}
    \mathcal{L}^\prime & \mathcal{Z} \\ \mathcal{Z}^\ast & \operatorname{Id}_{\Hi^{n-1}}
\end{bmatrix}.
\end{equation}
By construction, we can take the operator $\mathcal{C}:\Hi^{n-1}\to \Hi^{2n-1}$ given by
\begin{equation}\label{eq:C}
\mathcal{C}:=\begin{bmatrix}
    \mathcal{Z} \\ \operatorname{Id}_{\Hi^{n-1}}
\end{bmatrix}
\end{equation}
as an onto decomposition of $\mathcal{M}=\mathcal{C}\mathcal{C}^\ast$.

\paragraph{The operator \texorpdfstring{$\mathcal{A}$}{A}:}

Set $\mathcal{A}_D:=\operatorname{diag}(A_1,\ldots,A_n)$, that is,
$$\mathcal{A}_D(\mathbf{x})=\left(A_1(x_1),\ldots,A_n(x_n)\right),\quad\forall \mathbf{x}=(x_1,\ldots,x_n)\in\Hi^n.$$
Making use of the Definition \ref{def: P Q}, denote
$$\mathcal{P}:=P\left(\overline{G^\prime}\right)\otimes \operatorname{Id}_\Hi\quad\text{and}\quad \mathcal{Q}:=Q(G^\prime)\otimes\operatorname{Id}_\Hi,$$
where $\overline{G^\prime}$ is the complementary subgraph of $G^\prime$, namely, $\overline{G^\prime}:=(\mathcal{N},\mathcal{E}\setminus\mathcal{E}^\prime)$. Hence, given $\tau>0$, we define the operator $\mathcal{A}:\Hi^{2n-1}\rightrightarrows \Hi^{2n-1}$ as
\begin{equation}\label{eq: A}
    \mathcal{A}:=\begin{bmatrix}
    \tau\mathcal{A}_D+\mathcal{P}+\mathcal{Q} & -\mathcal{Z} \\ \mathcal{Z}^\ast & \mathbf{0}_{\Hi^{n-1}}
\end{bmatrix}.
\end{equation}

\paragraph{The operator \texorpdfstring{$\mathcal{B}$}{B}:}

Let $\mathcal{R}:\Hi^n\to \Hi^n$ be the linear operator
\begin{equation*}\label{eq: R}
    \mathcal{R}:=P(G)\otimes\operatorname{Id}_\Hi,
\end{equation*}
and define $\mathcal{B}_D:\Hi^n\to \Hi^n$ to be
\begin{equation*}\label{eq: BD}
    \mathcal{B}_D:=\operatorname{diag}(0,B_1,\ldots,B_{n-1})\left(\operatorname{Adj}(G^{\prime\prime})^\ast\otimes\Id_{\Hi}\right).
\end{equation*}
More explicitly, $\mathcal{B}_D(\mathbf{x})=\left[0,B_1\left(x_{\pre(2)}\right),\ldots,B_{n-1}\left(x_{\pre(n)}\right)\right]$ for $\mathbf{x}=[x_1,\ldots,x_n]\in\Hi^n$, where $\pre$ is given in~\ref{axi: alg subtree} of Assumption~\ref{assgraphs}. Finally, define the operator $\mathcal{B}:\Hi^{2n-1}\to \Hi^{2n-1}$  as
\begin{equation}\label{eq: B}
    \mathcal{B}:=\operatorname{diag}(\tau\mathcal{B}_D+\tfrac{\tau}{4\beta}\mathcal{R}, \mathbf{0}_{\Hi^{n-1}}).
\end{equation}

\subsection{Properties of the operators}

As proved in \cite[Theoren 3.1]{graph-drs}, the set-valued operator $\mathcal{A}$ is maximally monotone. This is also the case for the operator $\mathcal{B}$, as shown next.

\begin{lemma}\label{lem: B maximally monotone}
    Under Assumption~\ref{assgraphs}, the operator $\mathcal{B}$ defined in \eqref{eq: B} is maximally monotone.
\end{lemma}

\begin{proof}
    It is clear that $\mathcal{B}$ is single-valued and continuous, since it is a combination of algebraic operations between cocoercive operators and linear mappings. Hence, if we prove that $\mathcal{B}$ is monotone, by Proposition \ref{prop: maximally monotone continuous}, it is maximally monotone. Take any $\mathbf{x},\mathbf{x}^\prime\in \Hi^{n}$ and let us denote $\Delta\mathbf{x}:=\mathbf{x}-\mathbf{x}^\prime\in \Hi^n$ and $\Delta\mathbf{b}:=\mathcal{B}_D(\mathbf{x})-\mathcal{B}_D(\mathbf{x}^\prime)$. Then, we need to prove that
    \begin{equation}\label{eq: B monotone}
    \left\langle\tau\Delta\mathbf{b}+\frac{\tau}{4\beta}\mathcal{R}(\Delta\mathbf{x}),\Delta\mathbf{x}\right\rangle\geq0.
    \end{equation}
    Denoting $\Delta B_{ij}:=B_i(x_j)-B_i(x^\prime_j)$, we get that
    \begin{align*}
        \langle\Delta\mathbf{b},\Delta\mathbf{x}\rangle&=\sum_{(j,i)\in\mathcal{E}^{\prime\prime}}\langle B_i(x_j)-B_i(x_j^\prime),x_i-x_i^\prime\rangle\\
        &=\sum_{(j,i)\in\mathcal{E}^{\prime\prime}}\langle \Delta B_{ij},\Delta x_i\rangle\\
        &=\sum_{(j,i)\in\mathcal{E}^{\prime\prime}}\left(\langle \Delta B_{ij},\Delta x_i-\Delta x_j\rangle+\langle \Delta B_{ij},\Delta x_j\rangle\right)\\
        &\geq\sum_{(j,i)\in\mathcal{E}^{\prime\prime}}\left(\langle \Delta B_{ij},\Delta x_i-\Delta x_j\rangle+\beta\norm{\Delta B_{ij}}^2\right).
    \end{align*}
    Further, by definition of $\mathcal{R}$, we obtain
    \begin{align*}
        \langle\mathcal{R}(\Delta\mathbf{x}),\Delta\mathbf{x}\rangle&=\sum_{i=1}^n\left(d_i\norm{\Delta x_i}^2+\sum_{(j,i)\in\mathcal{E}}-2\langle \Delta x_i,\Delta x_j\rangle\right) \\
        &=\sum_{(j,i)\in\mathcal{E}}\left(\norm{\Delta x_i}^2-2\langle \Delta x_i,\Delta x_j\rangle+\norm{\Delta x_j}^2\right) \\
        &=\sum_{(j,i)\in\mathcal{E}}\norm{\Delta x_i-\Delta x_j}^2.
    \end{align*}
    Gathering both expressions, we get
    \begin{align*}
        \left\langle\Delta\mathbf{b}+\frac{1}{4\beta}\mathcal{R}(\Delta\mathbf{x}),\Delta\mathbf{x}\right\rangle \geq~& \sum_{(j,i)\in\mathcal{E}^{\prime\prime}}\langle \Delta B_{ij},\Delta x_i-\Delta x_j\rangle\\
        & +\beta\norm{\Delta B_{ij}}^2+\frac{1}{4\beta}\sum_{(j,i)\in\mathcal{E}}\norm{\Delta x_j-\Delta x_i}^2\\
        =~ & \sum_{(j,i)\in\mathcal{E}^{\prime\prime}}\norm{\sqrt{\beta}\Delta B_{ij}+\frac{1}{2\sqrt{\beta}}\left(\Delta x_j-\Delta x_i\right)}^2\\
        & +\frac{1}{4\beta}\sum_{(j,i)\in\mathcal{E}\setminus\mathcal{E}^{\prime\prime}}\norm{\Delta x_j-\Delta x_i}^2.
    \end{align*}
    Hence~\eqref{eq: B monotone} holds, which proves that $\mathcal{B}$ is monotone, as desired.
\end{proof}

\begin{remark} Contrarily to $\mathcal{A}$, where the maximally monotone operators $A_1,\ldots,A_n$ give rise to another maximally monotone operator, the cocoercivity is no inherited by $\mathcal{B}$.
For simplicity, let $n=2$, so there is only one graph setting with two nodes. Take $B_1:=\Id_\Hi$. Thus, the operator $\mathcal{B}$ has the form
\[\mathcal{B}=\frac{\tau}{4}\left[\begin{array}{cc;{2pt/2pt}c}
    \operatorname{Id}_\Hi & 0 & 0 \\
    2\operatorname{Id}_\Hi & \operatorname{Id}_\Hi & 0 \\ \hdashline[2pt/2pt]
    0 & 0 & 0
\end{array}\right].\]
Pick any $x\in\Hi\setminus\{0\}$ and set $\mathbf{x}:=[x,0,0],\mathbf{x}^\prime:=[0,x,0]\in\Hi^3$. Then
$\langle \mathcal{B}(\mathbf{x})- \mathcal{B}(\mathbf{x}^\prime),\mathbf{x}-\mathbf{x}^\prime\rangle=0$, while $\|\mathcal{B}(\mathbf{x})- \mathcal{B}(\mathbf{x}^\prime)\|^2=2\|x\|^2$, so $\mathcal{B}$ is not cocoercive.
\end{remark}

The next result relates the set of zeros of the operator $\mathcal{A}+\mathcal{B}$ with that of the sum of the original operators. It is similar to~\cite[Theorem~3.1]{graph-drs} but
incorporates the operator $\mathcal{B}$ and the cocoercive operators $B_1,\ldots,B_{n-1}$, so we include its proof for completeness.

\begin{theorem}\label{th: zeros maximally monotone}
    Suppose that Assumption~\ref{assgraphs} holds. Given $\tau>0$, let $\mathcal{A}$ and $\mathcal{B}$ be the operators defined in \eqref{eq: A} and \eqref{eq: B}. Then, for all $\mathbf{x}=[x_1,\ldots,x_n]\in \Hi^n$, it holds that
    \[\exists\mathbf{v}\in \Hi^{n-1} \text{ such that }[\mathbf{x},\mathbf{v}]\in\operatorname{zer}(\mathcal{A}+\mathcal{B})\iff x_1=\cdots=x_n\in\operatorname{zer}\left(\sum_{i=1}^n A_i+ \sum_{i=1}^{n-1}B_i\right).\]
    Moreover, the operator $\mathcal{A}+\mathcal{B}$ is maximally monotone.
\end{theorem}

\begin{proof}
    Let $\mathbf{x}\in \Hi^{n}$ and suppose that there exist some $\mathbf{v}\in \Hi^{n-1}$ such that $[\mathbf{x},\mathbf{v}]\in\operatorname{zer}(\mathcal{A}+\mathcal{B})$. By construction of the operators $\mathcal{A}$ and $\mathcal{B}$ this is equivalent to
        \begin{align*}
            \mathbf{0}_{\Hi^{n}} & \in\left(\tau\mathcal{A}_D + \tau\mathcal{B}_D + \mathcal{P}+\mathcal{Q}+\tfrac{\tau}{4\beta}\mathcal{R}\right)(\mathbf{x})-\mathcal{Z}\mathbf{v},\\
            \mathbf{0}_{\Hi^{n-1}} & =\mathcal{Z}^\ast\mathbf{x}.
        \end{align*}
    From the second equation, since $\ker Z^\ast=\operatorname{span}\{\mathbf{1}\}$, we easily obtain that ${x_1=\cdots=x_n=:x}$. On the other hand, the first equation implies the existence of $a_i\in A_i(x)$, for $i=1,\ldots,n$, and $b_i= B_i(x)$, for $i=1,\ldots,n-1$, such that
    \begin{equation}\label{eq: a+b+Q+P+Z}
\tau\mathbf{a}+\tau\mathbf{b}+\mathcal{Q}\mathbf{x}+\mathcal{P}\mathbf{x}+\frac{\tau}{4\beta}\mathcal{R}\mathbf{x}-\mathcal{Z}\mathbf{v}=\mathbf{0}_{\Hi^{n}},
    \end{equation}
    where $\mathbf{a}:=[a_1,\ldots,a_n]$ and $\mathbf{b}:=[0,b_1,\ldots,b_{n-1}]$.
    Now, we operate by $(\mathbf{1}\otimes \operatorname{Id}_\Hi)^\ast=\mathbf{1}^\ast\otimes\operatorname{Id}_\Hi$  on the left of equation~\eqref{eq: a+b+Q+P+Z} and calculate the summands. First, we obtain that
    \[(\mathbf{1}\otimes \operatorname{Id}_\Hi)^\ast(\tau\mathbf{a}+\tau\mathbf{b})=\tau\sum_{i=1}^n a_i+\tau\sum_{i=1}^{n-1}b_i\in\tau \left(\sum_{i=1}^n A_i+ \sum_{i=1}^{n-1}B_i\right)(x).\]
    Noting that $\mathbf{x}=(\mathbf{1}\otimes\operatorname{Id}_\Hi)x$, the remaining terms can be written as
    \begin{align*}
    (&\mathbf{1}\otimes \operatorname{Id}_\Hi)^\ast\left(\mathcal{Q}\mathbf{x}+\mathcal{P}\mathbf{x}+\tfrac{\tau}{4\beta}\mathcal{R}\mathbf{x}-\mathcal{Z}\mathbf{v}\right)\\
    &=\left(\mathbf{1}^\ast Q(G^\prime)\mathbf{1}\otimes \operatorname{Id}_\Hi\right)x+\left(\mathbf{1}^\ast P\left(\overline{G^\prime}\right)\mathbf{1}\otimes \operatorname{Id}_\Hi\right)x+\tfrac{\tau}{4\beta}\left(\mathbf{1}^\ast P(G)\mathbf{1}\otimes \operatorname{Id}_\Hi\right)x-\left(\mathbf{1}^\ast Z\mathbf{1}\otimes \operatorname{Id}_\Hi\right)\mathbf{v}.
    \end{align*}
    All terms in the previous expression are zero. Indeed, by Remark \ref{rem: P laplacian}, the matrix $Q(G^\prime)$ is skew-symmetric, so $\mathbf{1}^\ast Q(G^\prime)\mathbf{1}=0$.
    Now, let us note that
    $$\mathbf{1}^\ast P(G)\mathbf{1}=\mathbf{1}^\ast \left(\frac{1}{2}\left(P(G)+P(G)^\ast\right)\right)\mathbf{1}=\mathbf{1}^\ast \operatorname{Lap}(G)\mathbf{1}=0,$$
    since  $\mathbf{1}\in \ker(\operatorname{Lap}(G))$, according to Lemma~\ref{lem: rank incidence}. The same argument applies to $P(G')$ and also to $Z$, since $\mathbf{1}\in \ker{Z}^\ast$ by Proposition~\ref{prop: Zdecom}.

    Therefore, putting all the above computations together, we conclude that
    \begin{equation*}\label{eq: ai plus bi zero}
    0=\sum_{i=1}^n a_i+\sum_{i=1}^{n-1}b_i\in\left(\sum_{i=1}^n A_i+ \sum_{i=1}^{n-1}B_i\right)(x).
    \end{equation*}

    For the reverse implication, suppose that $x\in\operatorname{zer}\left(\sum_{i=1}^n A_i+ \sum_{i=1}^{n-1}B_i\right)$, i.e.,
    \begin{equation}\label{eq: ai plus bi zero2}
    0 = \sum_{i=1}^n a_i + \sum_{i=1}^{n-1} b_i,
    \end{equation}
    with $a_i\in A_i(x)$, for $i=1,\ldots,n$, and $b_i= B_i(x)$, for $i=1,\ldots,n-1$. Let $\mathbf{x}:=(\mathbf{1}\otimes\operatorname{Id}_\Hi)x\in \Hi^{n}$, so that one trivially has $\mathcal{Z}^\ast\mathbf{x}=0$. It thus suffices to find $\mathbf{v}\in \Hi^{n-1}$ satisfying equation~\eqref{eq: a+b+Q+P+Z} or, equivalently,
    \begin{equation}\label{eq: zero equiv ker}
    \tau\mathbf{a}+\tau\mathbf{b}+\mathcal{Q}\mathbf{x}+\mathcal{P}\mathbf{x}+\tfrac{\tau}{4\beta}\mathcal{R}\mathbf{x}\in\operatorname{Im}\mathcal{Z}=(\ker \mathcal{Z}^\ast)^\perp,
    \end{equation}
    where $\mathbf{a}:=[a_1,\ldots,a_n]$ and $\mathbf{b}:=[0,b_1,\ldots,b_{n-1}]$. Let us see that this inclusion holds. Since $\ker \mathcal{Z}^\ast=\{[x,\ldots,x] : x\in\Hi\}=\operatorname{ran}(\mathbf{1}\otimes\operatorname{Id}_\Hi)=\ker (\mathbf{1}\otimes\operatorname{Id}_\Hi)^\ast$,
    then \eqref{eq: zero equiv ker} holds if and only if
    $$(\mathbf{1}\otimes\operatorname{Id}_\Hi)^\ast\left(\tau\mathbf{a}+\tau\mathbf{b}+\mathcal{Q}\mathbf{x}+\mathcal{P}\mathbf{x}+\tfrac{\tau}{4\beta}\mathcal{R}\mathbf{x}\right)=0,$$
    which holds by~\eqref{eq: ai plus bi zero2} and the same argumentation as in the first part of the proof.

    Finally, to prove that $\mathcal{A}+\mathcal{B}$ is maximally monotone, recall that $\mathcal{A}$ and $\mathcal{B}$ are maximally monotone by \cite[Theoren 3.1]{graph-drs} and Lemma~\ref{lem: B maximally monotone}, respectively. Since $\operatorname{dom}\mathcal{B}=\Hi^{2n-1}$, the maximal monotonicity of the sum $\mathcal{A}+\mathcal{B}$ follows from Proposition~\ref{prop: sum of maximally monotone}.
\end{proof}

\begin{lemma}\label{lem: admissible preconditioner}
Suppose that Assumption~\ref{assgraphs} holds and let $Z\in\mathbb{R}^{n\times(n-1)}$ be an onto decomposition of $\operatorname{Lap}(G^\prime)$. Given $\tau>0$, let $\mathcal{M}$, $\mathcal{A}$ and $\mathcal{B}$ be the operators defined in~\eqref{eq: M}, \eqref{eq: A} and \eqref{eq: B}. Then $\mathcal{M}$ is an admissible preconditioner for $\mathcal{A}+\mathcal{B}$. Further, the operator $(\mathcal{M}+\mathcal{A}+\mathcal{B})^{-1}$ is Lipschitz continuous.
\end{lemma}

\begin{proof}
Let us denote $\mathcal{A}_L:=\tau\mathcal{A}_D+\mathcal{P}+\mathcal{Q}$ and $\mathcal{B}_L:=\tau\mathcal{B}_D+\frac{\tau}{4\beta}\mathcal{R}$. Then, one has
\begin{align}
    \begin{bmatrix} {\mathbf{x}} \\ {\mathbf{v}}\end{bmatrix}&\in(\mathcal{M}+\mathcal{A}+\mathcal{B})^{-1}\begin{bmatrix} \mathbf{z} \\ \mathbf{y}\end{bmatrix} \iff\begin{bmatrix} \mathbf{z} \\ \mathbf{y}\end{bmatrix} \in     \begin{bmatrix}
   \mathcal{L}^\prime + \mathcal{A}_L + \mathcal{B}_L & \mathbf{0}_{\Hi^n}\\ 2\mathcal{Z}^\ast & \operatorname{Id}_{\Hi^n}
\end{bmatrix}\begin{bmatrix} {\mathbf{x}} \\ {\mathbf{v}} \end{bmatrix}\nonumber\\
&\iff\left\{\begin{array}{l}
\mathbf{z}\in(\mathcal{L}^\prime+\mathcal{P}+\mathcal{Q}+\tfrac{\tau}{4\beta}\mathcal{R}+\tau\mathcal{A}_D+\tau\mathcal{B}_D)(\mathbf{x}),\\
\mathbf{y}=2\mathcal{Z}^\ast{\mathbf{x}}+{\mathbf{v}}.
\end{array}\right.\label{eq: inverse M plus A plus B}
\end{align}
Now, taking into account the definition of the operators involved, we have that
\begin{align*}
\mathcal{L}^\prime+\mathcal{P}+\mathcal{Q}+\tfrac{\tau}{4\beta}\mathcal{R} &
=\left(\operatorname{Lap}(G^\prime)+P\left(\overline{G^\prime}\right)+Q\left({G^\prime}\right)+\tfrac{\tau}{4\beta} P\left({G}\right)\right)\otimes\operatorname{Id}_\Hi\\
& =\left(P\left(\overline{G^\prime}\right)+P\left({G^\prime}\right)+\tfrac{\tau}{4\beta} P\left({G}\right)\right)\otimes\operatorname{Id}_\Hi\\
& =(1+\tfrac{\tau}{4\beta})P\left({G}\right)\otimes\operatorname{Id}_\Hi.
\end{align*}
Combining this with the first inclusion in~\eqref{eq: inverse M plus A plus B} yields
\begin{equation*}
\mathbf{z} \in  \left((1+\tfrac{\tau}{4\beta})P\left({G}\right)\otimes\operatorname{Id}_\Hi \right)(\mathbf{x}) + \tau\mathcal{A}_D(\mathbf{x})+\tau\mathcal{B}_D(\mathbf{x}).
\end{equation*}
Analyzing this expression componentwise, we arrive at
\begin{equation}\label{eq:ztou}
\begin{aligned}
 z_1  \in~ & (1+\tfrac{\tau}{4\beta})d_1x_1 + \tau A_1(x_1),\\
 z_i  \in~ & (1+\tfrac{\tau}{4\beta})\left(d_i x_i - 2\sum_{(h,i)\in\mathcal{E}}x_h\right) + \tau A_i(x_i) + \tau B_{i-1}(x_{\pre(i)}),\quad\text{for }i=2,\ldots,n,
\end{aligned}
\end{equation}
where $\pre(i)$ is the unique node such that $(\pre(i),i)\in\mathcal{E}^{\prime\prime}$ (recall~\ref{axi: alg subtree} in Assumption~\ref{assgraphs}).
Then, dividing each inclusion by $(1+\tfrac{\tau}{4\beta})d_i$, letting $\gamma:=(1+\tfrac{\tau}{4\beta})^{-1}\tau$ and rearranging, we deduce that $\textbf{u}$ and $\textbf{v}$ in~\eqref{eq: inverse M plus A plus B}  are uniquely determined by
\begin{equation}\label{eq:utoz resolvents}
\left\{\begin{aligned}
x_1&=J_{\frac{\gamma}{d_1}A_1}\left(\frac{\gamma}{\tau d_1}z_1\right),\\
x_i&=J_{\frac{\gamma}{d_i}A_i}\left(\frac{2}{d_i}\sum_{(h,i)\in\mathcal{E}}x_h-\frac{\gamma}{d_i} B_i(x_{\pre(i)})+\frac{\gamma }{\tau d_i}z_i\right),\quad\text{for }i=2,\ldots,n,\\
\mathbf{v}&=\mathbf{y}-2\mathcal{Z}^\ast{\mathbf{x}}.
\end{aligned}\right.
\end{equation}
In particular, this implies that $(\mathcal{M}+\mathcal{A}+\mathcal{B})^{-1}$ is single-valued and Lipschitz, as it can be expressed as a composition of resolvents of maximally monotone operators, cocoercive operators and linear combinations. Finally, the fact that $\mathcal{M}$  is an admissible preconditioner for $\mathcal{A}+\mathcal{B}$ is a direct consequence of $J_{M^{-1}(\mathcal{A}+\mathcal{B})}=(\mathcal{M}+\mathcal{A}+\mathcal{B})^{-1}\mathcal{M}$.
\end{proof}

\section{A graph based forward-backward method}\label{sec:method}

This section is devoted to the construction and analysis of our main algorithm. After establishing its convergence, we generate
several instances of the scheme by considering different graph configurations.  Some of these coincide with or are related to some known methods, while others, as the ones generated by the complete graph, seem to be new and promising.

\subsection{Development and convergence of the method}\label{subsec:method}

Once defined the required operators and graph settings, we present the resulting method for solving \eqref{eq:Prob} in Algorithm~\ref{alg: our method}. Our main convergence result is given in Theorem~\ref{th: convergence}.

\begin{algorithm}[ht!]
    \caption{Graph based forward-backward method}
    \label{alg: our method}
    \begin{algorithmic}[1]
        \State \textbf{let:} $w_1^0,\ldots,w_{n-1}^0\in\mathcal{\Hi}$ and some parameters $\gamma\in{]0,4\beta[}$ and $\theta_k \in {] 0 ,2-\gamma/(2\beta)]}$.
        \For {$k=0,1,2,\ldots$}
        \State {$x_1^{k+1}=J_{\frac{\gamma}{d_1}A_1}\left(\frac{1}{d_1}\sum_{j=1}^{n-1}Z_{1j}w^k_j\right)$}
        \State {$x_i^{k+1}=J_{\frac{\gamma}{d_i}A_i}\left(\frac{2}{d_i}\sum_{(h,i)\in\mathcal{E}}x_h^{k+1}-\frac{\gamma}{d_i}B_{i-1}(x_{\pre(i)}^{k+1})+\frac{1}{d_i}\sum_{j=1}^{n-1}Z_{ij}w^k_j\right)$} \Comment{$i\in \llbracket 2,n\rrbracket$}
        \State {$w_i^{k+1}=w_i^k-\theta_k\sum_{j=1}^n Z_{ji}x_j^{k+1}$} \Comment{$i\in \llbracket1,n-1\rrbracket$}
        \EndFor
    \end{algorithmic}
\end{algorithm}

\begin{theorem}[Convergence of Algorithm~\ref{alg: our method}]\label{th: convergence}
    Suppose that Assumption~\ref{assgraphs} holds and let $Z\in\mathbb{R}^{n\times(n-1)}$ be an onto decomposition of the Laplacian of $G^\prime$. Pick any $w_1^0,\ldots,w_{n-1}^0\in\mathcal{\Hi}$ and let $\{w_j^k\}_{k=0}^\infty$ and $\{x_i^{k+1}\}_{k=0}^\infty$ be the sequences generated by Algorithm \ref{alg: our method} with stepsize $\gamma\in{]0,4\beta[}$ and relaxation parameters $\{\theta_k\}_{k=0}^\infty$ satisfying
\begin{equation}\label{eq:sumthetak_alg}
\theta_k \in {\left] 0 ,\tfrac{4\beta-\gamma}{2\beta}\right]} \quad\text{and}\quad \sum_{k=0}^{\infty}\theta_k\left( \tfrac{4\beta-\gamma}{2\beta} -\theta_k\right) = +\infty.
\end{equation}
Then, the following assertions hold:
    \begin{enumerate}[label=(\roman*)]
        \item $w_j^k\rightharpoonup w_j^\ast$ for some $w_j^\ast\in\Hi$, for $j=1,\ldots,n-1$;
        \item $x_i^{k+1}\rightharpoonup x^\ast\in\zer\left(\sum_{i=1}^n A_i+\sum_{j=1}^{n-1}B_j\right)$, for all $i=1,\ldots,n$, with
        \begin{align*}
        x^\ast:=& J_{\frac{\gamma}{d_1}A_1}\left(\frac{1}{d_1}
        \sum_{j=1}^{n-1}Z_{ij}w_j^\ast\right)\\
        = & J_{\frac{\gamma}{d_i}A_i}\left(\frac{2d_i^{\rm in}}{d_i}x^\ast-\frac{\gamma}{d_i}B_i(x^\ast)+\frac{1}{d_i}\sum_{j=1}^{n-1}Z_{ij}w^\ast_j\right),\quad \text{for all } i=2,\ldots,n.
        \end{align*}
    \end{enumerate}
\end{theorem}
\begin{proof}
The proof is an adaptation of that of~\cite[Theorem 3.2]{graph-drs}, taking into consideration the new additions related to the inclusion of cocoercive operators. In this way, we shall rewrite Algorithm~\ref{alg: our method} as an instance of \eqref{eq: reduced preconditioned proximal point}. To this aim, set
$$ \tau:=\tfrac{4\beta}{4\beta-\gamma} \gamma >0$$
and consider the operators $\mathcal{A}$ and $\mathcal{B}$ respectively defined in \eqref{eq: A} and \eqref{eq: B}. Define $\mathcal{M}$ by~\eqref{eq: M}, which has an onto decomposition $\mathcal{M}=\mathcal{C}\mathcal{C}^\ast$, with $\mathcal{C}$ given by~\eqref{eq:C}, and let
\begin{equation*}
\mu_k:=\tfrac{4\beta}{4\beta-\gamma}\theta_k,\quad \text{for each } k=0,1,\ldots.
\end{equation*}

Note that the sequence $\{\mu_k\}_{k=0}^\infty$ verifies \eqref{def: relaxation paramter} in view of \eqref{eq:sumthetak_alg}. Hence, thanks to Lemmas~\ref{lem: admissible preconditioner} and~\ref{lem: parallel resolvent}, we can apply~\eqref{eq: reduced preconditioned proximal point} using $\{\mu_k\}_{k=0}^\infty$ as relaxation parameters. Thus, given any starting point $\mathbf{y}^0\in\Hi^{n-1}$, this gives rise to the sequence
\begin{equation}\label{eq: RPPP in our method}
\mathbf{y}^{k+1}=\mathbf{y}^k+\mu_k(J_{\mathcal{C}^\ast\rhd(\mathcal{A}+\mathcal{B})}(\mathbf{y}^k)-\mathbf{y}^k),\quad k=0,1,\ldots,
\end{equation}
with $J_{\mathcal{C}^\ast\rhd(\mathcal{A}+\mathcal{B})}(\mathbf{y}^k)=\mathcal{C}^\ast(\mathcal{M}+\mathcal{A}+\mathcal{B})^{-1}(\mathcal{C}\mathbf{y}^k)$.
Observe that, as in Lemma~\ref{lem: admissible preconditioner}, it holds $(1+\tfrac{\tau}{4\beta})^{-1}\tau=\gamma$, so if we let
\begin{equation}\label{eq:xk first N}
\begin{bmatrix} \mathbf{x}^{k+1} \\ \mathbf{v}^{k+1}\end{bmatrix}:=(\mathcal{M}+\mathcal{A}+\mathcal{B})^{-1}(\mathcal{C}\mathbf{y}^k)=(\mathcal{M}+\mathcal{A}+\mathcal{B})^{-1}\begin{bmatrix} \mathcal{Z}\mathbf{y}^{k} \\ \mathbf{y}^k\end{bmatrix},
\end{equation}
then~\eqref{eq:utoz resolvents} gives
\begin{equation}\label{eq:x comput}
\left\{\begin{aligned}
x_1^{k+1}&=J_{\frac{\gamma}{d_1}A_1}\left(\frac{\gamma}{\tau d_1}\sum_{j=1}^{n-1}Z_{1j}y^k_j\right),\\
x_i^{k+1}&=J_{\frac{\gamma}{d_i}A_i}\left(\frac{2}{d_i}\sum_{(h,i)\in\mathcal{E}}x_h^{k+1}-\frac{\gamma}{d_i} B_i(x_{\pre(i)}^{k+1})+\frac{\gamma }{\tau d_i}\sum_{j=1}^{n-1}Z_{ij}y^k_j\right),\quad\text{for }i=2,\ldots,n,\\
\mathbf{v}^{k+1}&=\mathbf{y}^k-2\mathcal{Z}^\ast{\mathbf{x}^{k+1}}.
\end{aligned}\right.
\end{equation}
Hence,
\begin{equation*}
J_{\mathcal{C}^\ast\rhd(\mathcal{A}+\mathcal{B})}(\mathbf{y}^k)=\mathcal{C}^\ast\begin{bmatrix} \mathbf{x}^{k+1} \\ \mathbf{v}^{k+1}\end{bmatrix} = \mathcal{C}^\ast\begin{bmatrix}
    \mathbf{x}^{k+1} \\ \mathbf{y}^k-2\mathcal{Z}^\ast{\mathbf{x}^{k+1}}
\end{bmatrix} = \mathbf{y}^k-\mathcal{Z}^\ast{\mathbf{x}^{k+1}},
\end{equation*}
so~\eqref{eq: RPPP in our method} becomes
\begin{equation*}
\mathbf{y}^{k+1}=\mathbf{y}^k-\mu_k\mathcal{Z}^\ast\mathbf{x}^{k+1},\quad k=0,1,\ldots.
\end{equation*}
Multiplying this expression by $\frac{\gamma}{\tau}$ and making the change of variable $\mathbf{w}^k:=\frac{\gamma}{\tau}\mathbf{y}^k$, we precisely obtain Algorithm~\ref{alg: our method}, since $\theta_k=\frac{\gamma}{\tau}\mu_k$.

Finally, to prove the convergence statements, observe that we are under the setting of Theorem~\ref{th: RPPP convergence}, so we deduce that the iterative scheme~\eqref{eq: RPPP in our method} weakly converges, satisfying $\mathbf{y}^k\wto \mathbf{y}^\ast\in\Hi^{2n-1}$ and $(\mathcal{M}+\mathcal{A}+\mathcal{B})^{-1}(\mathcal{C}\mathbf{y}^k)\wto \mathbf{u}^\ast\in\zer (\mathcal{A}+\mathcal{B})$ with
    \begin{equation}\label{eq:ulim}
    \mathbf{u}^\ast:=(\mathcal{M}+\mathcal{A}+\mathcal{B})^{-1}(\mathcal{C}\mathbf{y}^\ast).
    \end{equation}
    In view of Theorem~\ref{th: zeros maximally monotone}, $\mathbf{u}^\ast=[\mathbf{x}^*, \mathbf{v}^\ast]$ for some $\mathbf{v}^\ast\in\Hi^{n-1}$ and $\mathbf{x}^\ast=[x^\ast,\ldots,x^\ast]\in\Hi^{n}$ with $x^\ast\in\operatorname{zer}\left(\sum_{i=1}^nA_i + \sum_{i=1}^{n-1}B_i\right)$. From \eqref{eq:xk first N} we note that $\mathbf{x}^{k+1}$ contains the first $n$ components of $(\mathcal{M}+\mathcal{A}+\mathcal{B})^{-1}(\mathcal{C}\mathbf{y}^k)$, so it holds that
    \begin{equation}\label{eq:xlim}
    \mathbf{x}^{k+1}\rightharpoonup \mathbf{x}^\ast.
    \end{equation}
    Rewriting \eqref{eq:ulim} and \eqref{eq:xlim} componentwise, and having in mind our change of variable, the result follows.
\end{proof}

\subsection{Some known instances of the algorithm}\label{subsec: known methods}

In the next examples we show how Algorithm~\ref{alg: our method} encompasses some other forward-backward methods in the literature as particular cases.

\begin{example}[Ring forward-backward]\label{ex: ring}
    Take the graphs $(G,G^\prime,G^{\prime\prime})$ as in Example~\ref{ex:graph_ring}. We get $d_i=2$ for all $i=0,\ldots,n$. Since $G^\prime$ is a tree, we can choose $Z=\operatorname{Inc}(G^\prime)$. In this setting, $Z_{ii}=1$ and $Z_{(i+1)i}=-1$ for all $i=1,\ldots,n-1$, while $Z_{ij}=0$ otherwise. Therefore, it can be verified that Algorithm \ref{alg: our method} is equivalent to the one shown in~\eqref{eq:ringFB}, originally developed in \cite{ring-networks}.
\end{example}

\begin{example}[Bredies--Chenchene--Naldi splitting]\label{ex: B=0}
    If we take $B_1=\ldots=B_{n-1}=0$, our problem reduces to a sum of $n$ maximally monotone operators and Algorithm \ref{alg: our method} recovers the method presented in {\cite[Algorithm~3.1]{graph-drs}}.
\end{example}

\begin{example}[Sequential FDR]\label{ex: sequential}
    Take $G$ as the sequential graph, see Example \ref{ex: sequential graph}. In this context, the only possible spanning subgraphs are $G^\prime=G^{\prime\prime}=G$. Thus, $d_1=d_n=1$ and $d_i=2$ for all $i=2,\ldots,n-1$. Moreover, since $G^\prime$ is a tree, then $Z=\operatorname{Inc}(G^\prime)$, which has the same configuration as in Example~\ref{ex: ring}. Then, Algorithm \ref{alg: our method}  takes de form
    \begin{equation}\label{eq:seqFB}
    \left\{ \begin{taligned}
    x_1^{k+1} &=  J_{\gamma A_1}\left(w_1^k\right),\\
    x_i^{k+1} &=   J_{\frac{\gamma}{2} A_i}\left(x_{i-1}^{k+1}-\frac{\gamma}{2} B_i(x_{i-1}^{k+1})+\frac{1}{2}\left({w_{i}^k-w_{i-1}^k}\right)\right), \quad \forall i\in \llbracket 2, n-1\rrbracket,\\
    x_n^{k+1} &=  J_{\gamma A_n}\left(2x_{n-1}^{k+1}-\gamma B_n(x_{n-1}^{k+1})-{w_{n-1}^k}\right),\\
    w_i^{k+1} &=   w_i^k+\theta_k\left(x_{i+1}^{k+1}-x_{i}^{k+1}\right), \qquad \forall i\in \llbracket 1, n-1\rrbracket,
    \end{taligned}\right.
    \end{equation}
    which is the \emph{sequential FDR (Forward Douglas--Rachford)} scheme presented in \cite[Eq.~(3.15)]{degenerate-ppp}.
\end{example}

\begin{example}[Parallel FDR]\label{ex: parallel}
    Take $G$ as the parallel up graph (see Example~\ref{ex: parallel graph}) and let $G^\prime=G^{\prime\prime}=G$. Then $d_1=n-1$ and $d_i=1$ for all $i=2,\ldots,n$. Also, since $G^\prime$ is a tree, we can choose $Z=\operatorname{Inc}(G^\prime)$. Thus, $Z_{(i+1)i}=-1$ and $Z_{1i}=1$ for all $i=1,\ldots,n-1$ and $Z_{ij}=0$ otherwise. In this case Algorithm~\ref{alg: our method} is expressed as
    \begin{equation}\label{eq:parFB}
    \left\{ \begin{taligned}
    x_1^{k+1} &=  J_{\frac{\gamma}{n-1} A_1}\left(\frac{1}{n-1}\sum_{j=1}^{n-1}w_j^k\right),\\
    x_i^{k+1} &=  J_{\gamma A_i}\left(2x_1^{k+1}-\gamma B_i(x_1^{k+1})-{w_{i-1}^k}\right), \quad \forall i\in \llbracket 2, n\rrbracket,\\
    w_i^{k+1} &=  w_i^k+\theta_k\left(x_{i+1}^{k+1}-x_1^{k+1}\right), \qquad \forall i\in \llbracket 1, n-1\rrbracket,
    \end{taligned}\right.
    \end{equation}
    This is the  \emph{parallel FDR} presented in \cite[Eq. (3.14)]{degenerate-ppp}.
\end{example}

\begin{example}[Four-operator splittings]\label{ex: 4 operator}
Let $n=3$, $B_1=0$ and let $B_2=B$ for a given $\beta$-cocoercive operator $B$. Choose $G$ as the complete graph (see Example~\ref{ex: complete graph}), so $d_i=2$ for $i=1,2,3$. Set $G'$ as the subgraph parallel down (Example~\ref{ex: parallel graph}) with edges $\mathcal{E}^\prime=\left\{(1,3),(2,3)\right\}$, which is a tree, so we can take $Z=\operatorname{Inc}(G^\prime)$. Under this setting, Algorithm~\ref{alg: our method} becomes
    \begin{equation}\label{eq:4operator}
    \left\{ \begin{taligned}
    x_1^{k+1} &=  J_{\frac{\gamma}{2} A_1}\left(\frac{1}{2}w_1^k\right),\\
    x_2^{k+1} &=  J_{\frac{\gamma}{2} A_2}\left(x_1^{k+1}+\frac{1}{2}w_{2}^k\right),\\
    x_3^{k+1} &=  J_{\frac{\gamma}{2} A_3}\left(x_1^{k+1}+x_2^{k+1}-\frac{\gamma}{2}B(x_{p(3)}^{k+1})-\frac{1}{2}w_{1}^k-\frac{1}{2}w_{2}^k\right),\\
    w_1^{k+1} &=  w_1^k+\theta_k\left(x_{3}^{k+1}-x_1^{k+1}\right),\\
    w_2^{k+1} &=  w_2^k+\theta_k\left(x_{3}^{k+1}-x_2^{k+1}\right).
    \end{taligned}\right.
    \end{equation}
    Making the change of variable $\mathbf{u}:=\frac{1}{2}\mathbf{w}$, $\lambda:=\frac{\gamma}{2}$ and $\eta_k:=\frac{\theta_k}{2}$ we precisely obtain the four-operator splittings introduced in~\cite{4operator}.
    Specifically, if we take $p(3)=1$, we capture~\cite[Algorithm~1]{4operator}, while~\cite[Algorithm~2]{4operator} is obtained if we set $p(3)=2$.
\end{example}

\subsection{Recovering a recent algorithm  as a limit case}

In \cite[Theorem 8.1]{MBG24}, the authors present a new frugal splitting algorithm with minimal lifting for solving~\eqref{eq:Prob} with $n\geq 2$ and $m\geq 0$.
Given $\gamma,\lambda,\mu>0$ such that
\begin{equation}\label{eq: condition MBG}
    \frac{\lambda}{2}\sum_{j=1}^m\frac{1}{\beta_j}<2-\mu(n-1),
\end{equation}
and $w_1^0,\ldots,w_{n-1}^0 \in\Hi$, their iterative scheme is defined by
\begin{equation*}
    \left\{
    \begin{taligned}
    x_1^{k+1}&=J_{\lambda A_1}\left(u_1^k\right), \\
    x_i^{k+1}&=J_{\frac{\lambda}{\mu}A_i}\left(x_1^{k+1}+\frac{1}{\mu}u_i^k\right), \quad \forall i\in\llbracket 2,n-1\rrbracket, \\
    y_j^{k+1}&=\lambda B_j(x_1^{k+1}), \quad \forall j\in\llbracket 1,m\rrbracket, \\
    \overline{y}^{k+1}&=\sum_{j=2}^{n-1}\left(u_j^k+\mu(x_1^{k+1}-x_j^{k+1})\right)+\sum_{j=1}^{m}y_j^{k+1}, \\
    x_n^{k+1}&=J_{\lambda A_n}\left(2x_1^{k+1}-u_1^k-\overline{y}^{k+1}\right), \\
    u_i^{k+1}&=u_i^{k}-\mu\left(x_i^{k+1}-x_n^{k+1}\right), \quad \forall i\in\llbracket 1,n-1\rrbracket.
    \end{taligned}
    \right.
\end{equation*}
Substituting the variables $y_j^{k+1}$ and $\overline{y}^{k+1}$ by their expression and denoting  $B:=\sum_{j=1}^m B_j$, it can be shortened as
\begin{equation}\label{alg: MBG}
    \left\{
    \begin{taligned}
    x_1^{k+1}&=J_{\lambda A_1}\left(u_1^k\right), \\
    x_i^{k+1}&=J_{\frac{\lambda}{\mu}A_i}\left(x_1^{k+1}+\frac{1}{\mu}u_i^k\right), \quad \forall i\in\llbracket 2,n-1\rrbracket, \\
    x_n^{k+1}&=J_{\lambda A_n}\left(\left(2-\mu(n-2)\right)x_1^{k+1} +\mu\sum_{i=2}^{n-1}x_i^{k+1}-\lambda B(x_1^{k+1})-\sum_{j=1}^{n-1}u_j^k\right),\\
    u_i^{k+1}&=u_i^{k}-\mu\left(x_i^{k+1}-x_n^{k+1}\right), \quad \forall i\in\llbracket 1,n-1\rrbracket.
    \end{taligned}
    \right.
    \end{equation}
Observe that, in virtue of \eqref{eq: condition MBG}, a necessary condition for $\mu$ is that $\frac{2}{n-1}<\mu$.

Let us show that an instance of Algorithm~\ref{alg: our method} defines an algorithm which can be interpreted as the limit case of \eqref{alg: MBG} when $\mu=\frac{2}{n-1}$. To this aim, we set all cocoercive operators in our framework to be zero, except for the last one $B_{n-1}$ which is set to $B$ (see Remark~\ref{rem: coco constant}). Recall that a cocoercivity constant of $B$ is given by $\beta:=(\sum_{j=1}^m \tfrac{1}{\beta_j})^{-1}$.

First, we must find a suitable triple $(G,G^\prime,G^{\prime\prime})$. The graph $G$ is determined by the appearances of the resolvent variables $x_i^{k+1}$ that update the next ones. By just observing this relation in~\eqref{alg: MBG}, we conclude that $(1,i+1),(i,n)\in\mathcal{E}$ for all $i=1,\ldots,n-1$, i.e., $G$ is the birapallel graph (see Example~\ref{ex: biparallel graph}).

On the other hand, for all $i=1,\ldots,n-1$, only $u_i^k$ is used to update $x_i^{k+1}$, yet $u_j^k$ for all $j=1,\ldots,n-1$ appears to update $x_n^{k+1}$. With this, we can deduce that $G^\prime$ is the parallel down graph, which is a tree, so we take $Z=\operatorname{Inc}(G^\prime)$. Thus, $Z_{ii}=1$ and $Z_{n,i}=-1$ for all $i=1,\ldots,n-1$.
Lastly, the cocoercive operator $B$ is evaluated only in $x_n^{k+1}$. This tells us that $(1,n)\in\mathcal{E}^{\prime\prime}$. However, since $G^{\prime\prime}$ must be a subgraph of $G$ whose nodes have a unique predecessor, the sole option for $G^{\prime\prime}$ is the parallel up.

Applying this graph setting to Algorithm~\ref{alg: our method} with $\gamma:=(n-1)\lambda$ (which requires $\lambda\leq\frac{4\beta}{n-1}$ by Theorem~\ref{th: convergence}) and making the change of variables $\mathbf{u}^k:=\frac{1}{n-1}\mathbf{w}^k$, we obtain the iterative scheme
\begin{equation*}
    \left\{
    \begin{taligned}
    x_1^{k+1}&=J_{\lambda A_1}\left(u_1^k\right), \\
    x_i^{k+1}&=J_{\frac{\lambda(n-1)}{2}A_i}\left(x_1^{k+1}+\frac{n-1}{2}u_i^k\right), \quad \forall i\in\llbracket 2,n-1\rrbracket, \\
    x_n^{k+1}&=J_{\lambda A_n}\left(\frac{2}{n-1}\sum_{i=1}^{n-1}x_i^{k+1}-\lambda B(x_1^{k+1})-\sum_{j=1}^{n-1}u_j^k\right),\\
    u_i^{k+1}&=u_i^{k}-\frac{\theta_k}{n-1}\left(x_i^{k+1}-x_n^{k+1}\right), \quad \forall i\in\llbracket 1,n-1\rrbracket.
    \end{taligned}
    \right.
    \end{equation*}
Observe that, if we were allowed to let $\mu\to\frac{2}{n-1}$ in~\eqref{alg: MBG}, we would exactly obtain the previous iterative scheme except for the relaxation parameter in the last equation, which would only be the same when $\theta_k=2$. Nevertheless, by~\eqref{eq:sumthetak_alg}, we can only choose $\theta_k\leq 2-\frac{(n-1)\lambda}{2\beta}<2$.
It remains as an open question for future investigation to expand our framework to fully cover the framework in~\cite{MBG24}.

\subsection{A forward-backward algorithm induced by the complete graph}\label{subsec: new method}

In this section, we derive the explicit iteration of an instance of Algorithm~\ref{alg: our method} in which we take $G=G'$ as the complete graph of order $n$ (see Example~\ref{ex: complete graph}). The particular case without cocoercive operators and $n=3$ was derived in~\cite{graph-drs}, where the authors showed promising numerical results. 

The Laplacian matrix of the complete graph $G'$ is given componentwise by
\begin{equation}\label{eq: Lap complete}
\operatorname{Lap}(G')_{ij}=\begin{cases}
    n-1 & \text{if }i=j,\\
    -1 & \text{otherwise}.
\end{cases}
\end{equation}
As shown in Proposition~\ref{prop: onto complete} in the Appendix, an onto decomposition $Z\in\R^{n\times(n-1)}$ of $\operatorname{Lap}(G')$ can be defined componentwise as
\begin{equation}\label{eq: Z complete}
Z_{ij}:=\begin{cases}
\sqrt{\frac{(n-i)n}{n-i+1}} & \text{if }i=j,\\
-\sqrt{\frac{n}{(n-j)(n-j+1)}} &\text{if } i>j,\\
0&\text{otherwise}.
\end{cases}
\end{equation}
Defining the new variable $\lambda:=\frac{\gamma}{n-1}$ and the constants
\begin{equation}\label{eq: a_i t_i}
a_i:=\sqrt{\frac{(n-i)n}{n-i+1}}\quad\text{and}\quad t_i:=-\sqrt{\frac{n}{(n-i)(n-i+1)}},\quad\text{for } i=1,\ldots,n-1, 
\end{equation}
we can rewrite Algorithm~\ref{alg: our method} as
\begin{equation*}\small
\left\{ \begin{taligned}
x_1^{k+1}&=J_{\lambda A_1}\left(\frac{1}{n-1}a_1w^k_1\right),\\
x_i^{k+1}&=J_{\lambda A_i}\left(\frac{2}{n-1}\sum_{h=1}^{i-1}x_h^{k+1}-\lambda B_{i-1}(x_{\pre(i)}^{k+1})+\frac{1}{n-1}\left(\sum_{j=1}^{i-1}t_jw^k_j+a_iw_i^k\right)\right), \quad \forall i\in \llbracket 2,n-1\rrbracket, \\
x_n^{k+1}&=J_{\lambda A_n}\left(\frac{2}{n-1}\sum_{h=1}^{n-1}x_h^{k+1}-\lambda B_{n-1}(x_{\pre(n)}^{k+1})+\frac{1}{n-1}\sum_{j=1}^{n-1}t_jw^k_j\right),  \\
w_i^{k+1}&=w_i^k-\theta_k\left(a_ix_i^{k+1}+t_i\sum_{j=i+1}^n x_j^{k+1}\right),\quad\forall i\in \llbracket1,n-1\rrbracket.
\end{taligned}\right.
\end{equation*}
The resulting algorithm has irrational coefficients, but can be simplified to rational ones. Indeed, by making the change of variables ${u}_j^k:=\frac{a_j}{n-1}w_j^k$ and $\mu_k:=\frac{n}{n-1}\theta_k$ we obtain the scheme shown in Algorihtm~\ref{alg: our method complete}, which we name the \emph{complete forward-backward method}.

\begin{algorithm}[ht!]
    \caption{Complete forward-backward method}
    \label{alg: our method complete}
    \begin{algorithmic}[1]
        \State \textbf{let:} $w_1^0,\ldots,w_{n-1}^0\in\mathcal{\Hi}$
        \For {$k=0,1,2,\ldots$}
        \State {$x_1^{k+1}=J_{\lambda A_1}\left(u^k_1\right)$}
        \State {$x_i^{k+1}=J_{\lambda A_i}\left(\frac{2}{n-1}\sum_{h=1}^{i-1}x_h^{k+1}-\lambda B_{i-1}(x_{\pre(i)}^{k+1})+u_i^k-\sum_{j=1}^{i-1}\frac{u^k_j}{n-j}\right)$} \Comment{$i\in \llbracket 2,n-1\rrbracket$}
        \State {$x_n^{k+1}=J_{\lambda A_n}\left(\frac{2}{n-1}\sum_{h=1}^{n-1}x_h^{k+1}-\lambda B_{n-1}(x_{\pre(n)}^{k+1})-\sum_{j=1}^{n-1}\frac{u^k_j}{n-j}\right)$}
        \State {$u_i^{k+1}=u_i^k-\mu_k\left(\frac{n-i}{n-i+1}x_i^{k+1}-\frac{1}{n-i+1}\sum_{j=i+1}^n x_j^{k+1}\right)$} \Comment{$i\in \llbracket1,n-1\rrbracket$}
        \EndFor
    \end{algorithmic}
\end{algorithm}

The convergence of Algorithm~\ref{alg: our method complete} can be directly deduced from Theorem~\ref{th: convergence} assuming that the stepsize $\lambda$ and the relaxation parameters $\{\mu_k\}_{k=0}^\infty$ satisfy
\begin{equation}\label{eq:sumthetak_alg_complete}
\lambda\in\left] 0,\tfrac{4\beta}{n-1} \right[,\; \mu_k \in {\left] 0 ,\left(\tfrac{2}{n-1}-\tfrac{\lambda}{2\beta}\right)n\right]} \;\text{ and }\; \sum_{k=0}^{\infty}\mu_k\left( \left(\tfrac{2}{n-1}-\tfrac{\lambda}{2\beta}\right)n -\mu_k\right) = +\infty.
\end{equation}

\section{Numerical experiment}\label{sec:experiments}

In this section, we present a numerical experiment to study how the graph setting affects the performance of the algorithm. In particular, we compare the algorithms presented in Section~\ref{subsec: known methods} with the complete forward-backward method proposed in Section~\ref{subsec: new method}. For this purpose, we consider the simple problem of minimizing $n-1$ convex quadratic functions (with global minimum at the origin) over the intersection of $n$ closed balls in $\mathcal{H}=\mathbb{R}^{200}$. Namely, the problem is
\begin{equation}\label{eq: minimizing problem constraints}
	\operatorname{Minimize}\sum_{j=1}^{n-1}\left(\frac{1}{2}x^T Q_jx\right)\text{ subject to } x\in\bigcap_{i=1}^nC_i,
\end{equation}
where $Q_j\in\mathbb{R}^{200\times 200}$ are positive semidefinite and $C_i:=\{x\in\mathbb{R}^{200}: \|x-c_i\|\leq r_i\}$ have a common intersection point in the interior, for $i=1,\ldots,n$ and $j=1,\ldots,n-1$.

Problem \eqref{eq: minimizing problem constraints} can be formulated as an inclusion problem of the form \eqref{eq:Prob}. Indeed, it can be modeled as
\begin{equation}\label{eq: inclusion experiment}
	\text{Find }x\in\R^{200} \text{ such that } 0\in\sum_{i=1}^nN_{C_i}(x)+\sum_{j=1}^{n-1}Q_jx,
\end{equation}
where $N_{C_i}$ is the normal cone to $C_i$, which is maximally monotone, and $Q_j$ is $\frac{1}{\norm{Q_j}_2}$-cocoercive, for $i=1,\ldots,n$ and $j=1,\ldots,n-1$. Although problem~\eqref{eq: inclusion experiment} can be simplified by letting $Q:=\sum_{j=1}^{n-1}Q_j$, our purpose is to test a distributed setting in which only $Q_j$ is known by node $j+1$, for $j=1,\ldots,n-1$.

\subsection{Description of the instances generated}\label{subsec: description experiment}

Let us explain the process that we follow to generate random instances of problem \eqref{eq: minimizing problem constraints}. The construction is conceived with the purpose of obtaining consistent problems with a nonempty feasible set not containing the origin (which is the minimizer of the unconstrained problem). We generate random initial points for the algorithms outside of the feasible set. The process is illustrated in Figure~\ref{fig: generation of balls}.

\paragraph{Quadratic functions} We generate a random matrix $W_j\in{[-0.5,0.5]}^{200\times 200}$ and define the positive semidefinite matrix $Q_j:=\frac{1}{2}W_j^TW_j$ for all $j=1,\ldots,n-1$. The cocoercivity constant is set to $\beta:=\min\{\norm{Q_1}_2^{-1},\ldots,\norm{Q_{n-1}}_2^{-1}\}$.

\newcommand{\feas}{z}
\paragraph{Feasible constraint sets}
We first take some random point $\feas\in{[-10,10]}^{200}$ and generate the centers of the balls around this point, which will be a point in the interior of all the sets. Specifically, each center $c_i\in\R^{200}$ is randomly generated so that
$$\norm{\feas-c_i}\in\left[\tfrac{\norm{\feas}}{6},\tfrac{\norm{\feas}}{3}\right], \text{ for all } i=1,\ldots,n.$$
Now, for each center $c_i$, $i=1,\ldots,n$, we pick a random value $\varepsilon_i\in{\left]0,\norm{\feas}/6\right[}$ and  set the correspondent radius as
$$r_i:=\norm{\feas-c_i}+\varepsilon_i.$$
Note that this radius is large enough to guarantee that the feasible point belongs to the intersection of the sets (as $r_i>\norm{\feas-c_i}$), yet small enough to exclude the origin, since
$$\norm{c_i}\geq\norm{z}-\norm{z-c_i}\geq2\norm{z-c_i}\geq \norm{z-c_i}+\frac{\norm{z}}{6}> \norm{z-c_i}+\epsilon_i=r_i.$$

\paragraph{Initial points for the algorithms} Having now the balls and the quadratic functions determined, we generate the variables $w_1^0,\ldots,w_{n-1}^0\in\R^{200}$ which will initiate the iterations. For simplicity, we take $w_1^0=\cdots=w_{n-1}^0=w^0$, with
$$w^0:=\feas+\left(\max_{i=1,\ldots,n}\{2r_i-\varepsilon_i\}+{\varepsilon}\right)\omega,$$
where $\omega\in\R^{200}$ is a random unitary vector and ${\varepsilon}$ is a random value in $[0,1]$. The point $w^0$ is set in this way so that $w^0\notin C_i$ for all $i=1,\ldots,n$, to avoid starting too close to the intersection.

\newcommand{\drawaxis}[2]{
	\draw [-stealth, line width=1] (0,#1) -- (0,#2);
	\draw [-stealth, line width=1] (#1,0) -- (#2,0);
}

\newcommand{\border}{
	\draw (-1,-1) rectangle (5,5);
	}

\newcommand{\ringdashed}[5]{
	\draw [dashed, #5!40!black] (#1,#2) circle [radius={#3}];
	\draw [dashed, #5!40!black] (#1,#2) circle [radius={#4}];
	\path [draw=none,fill=#5, fill opacity = 0.2, even odd rule] (#1,#2) circle ({#4}) (#1,#2) circle ({#3});
}

\newcommand{\centerballsdashed}[3]{
	\pgfmathsetmacro{\smallradius}{sqrt((#1-2)^2+(#2-2)^2)}
	\pgfmathsetmacro{\radius}{sqrt((#1-2)^2+(#2-2)^2)+#3}
	\draw [fill=red, draw=red!60!black, fill opacity=0.15] (#1,#2) circle [radius=\radius];
	\draw [line width=0.1, red, draw opacity = 0.6] (#1,#2) circle [radius=\smallradius];
}

\newcommand{\centerballs}[3]{
	\pgfmathsetmacro{\smallradius}{sqrt((#1-2)^2+(#2-2)^2)}
	\pgfmathsetmacro{\radius}{sqrt((#1-2)^2+(#2-2)^2)+#3}
	\draw [fill, red] (#1,#2) circle [radius=0.05];
	\draw [fill=red, draw=red!60!black, fill opacity=0.15] (#1,#2) circle [radius=\radius];
}

\begin{figure}[ht!]
	\centering
	\begin{subfigure}{0.3\textwidth}
		\centering
	\begin{tikzpicture}[scale=0.75]
		\drawaxis{-1}{5}
		\draw (0,0) -- (2,2);
		\draw [fill, blue] (2,2) circle [radius=0.075];
		\pgfmathsetmacro{\normf}{sqrt(2^2+2^2)}
		\ringdashed{2}{2}{\normf / 6}{\normf / 3}{blue}
		\draw [fill, red] (2.4,2.5) circle [radius=0.05];
		\draw [fill, red] (1.4,2) circle [radius=0.05];
		\draw [fill, red] (2.3,1.3) circle [radius=0.05];
	\end{tikzpicture}
	\caption{Generation of random centers (red) given a feasible point (blue)}
\end{subfigure}
\hfill
\begin{subfigure}{0.3\textwidth}
	\centering
	\begin{tikzpicture}[scale=0.75]
		\drawaxis{-1}{5}
		\centerballsdashed{2.4}{2.5}{0.45}
		\centerballsdashed{1.4}{2}{0.25}
		\centerballsdashed{2.3}{1.3}{0.1}
		\draw (2.4,2.5) -- (2,2);
		\draw (1.4,2) -- (2,2);
		\draw (2.3,1.3) -- (2,2);
		\draw [fill, red] (2.4,2.5) circle [radius=0.05];
		\draw [fill, red] (1.4,2) circle [radius=0.05];
		\draw [fill, red] (2.3,1.3) circle [radius=0.05];
		\draw [fill,blue] (2,2) circle [radius=0.075];
	\end{tikzpicture}
	\caption{From the centers, define the radii $r_i=\norm{\feas-c_i}+\epsilon_i$}
\end{subfigure}
\hfill
\begin{subfigure}{0.3\textwidth}
	\centering
	\begin{tikzpicture}[scale=0.75]
		\drawaxis{-1}{5}
		\ringdashed{2}{2}{1.7306}{2.7306}{green}
		\centerballs{2.4}{2.5}{0.45}
		\centerballs{1.4}{2}{0.25}
		\centerballs{2.3}{1.3}{0.1}
		\draw [fill,blue] (2,2) circle [radius=0.075];
		\draw [fill,green!50!black] (2,-0.5) circle [radius=0.075];
	\end{tikzpicture}
	\caption{Choice of the point $w_0$ (green) far enough from balls (green ring)}
\end{subfigure}
	\caption{Construction of the balls $C_i$ and the initial point $w_0$ in $\mathbb{R}^2$}
	\label{fig: generation of balls}
	\end{figure}
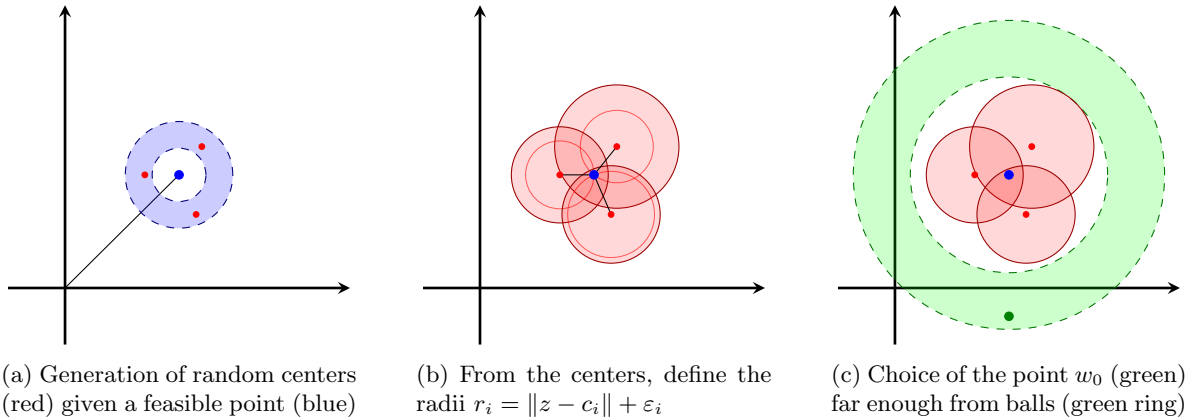

\subsection{Experiment setting and results}

In our experiment we compared the performance of the algorithms in Examples~\ref{ex: ring}, \ref{ex: sequential}~and~\ref{ex: parallel}, which will be referred to as \emph{ring}, \emph{sequential} and \emph{parallel} methods, respectively, as well as the new Algorithm~\ref{alg: our method complete}. For the latter, we considered two versions, depending on the choice of the second subgraph $G^{\prime\prime}$. Namely, we tested the sequential and the parallel-up graphs for~$G^{\prime\prime}$, so we refer to these algorithms as \emph{complete-seq} and \emph{complete-par}, respectively. For every algorithmic computation, we took the parameters
$$\gamma:={2\beta} \quad\text{and}\quad \theta_k:=0.99,\; \forall k\in\mathbb{N}.$$

For each $n\in\{3,\ldots,20\}$, we generated $10$ random problems as described above. For each problem, all the algorithms were run from the same $10$ random starting points. We computed both the iterations and the CPU running time required by each algorithms to achieve for the first time the tolerance error
\begin{equation*}
\max_{i=1,\ldots,n}\left\{\|x_i^{k+1}-x_i^k\|\right\} < 10^{-8}.
\end{equation*}
The tests were ran on a desktop of Intel Core i7-4770 CPU 3.40GHz with 32GB RAM, under Windows 10 (64-bit). The results are shown in Figure~\ref{fig: experiment}, where the colored shadows indicate the range along the $10$ problems, while the lines represent the median values.

\begin{figure}[htp!]\centering
	\begin{subfigure}{0.93\textwidth}
	\hspace{5.5mm}\includegraphics[width=\textwidth]{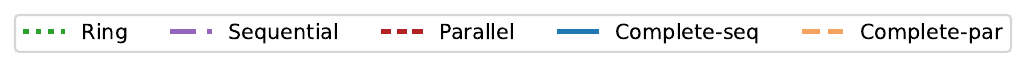}
	\end{subfigure}
	\begin{subfigure}{\ancho\textwidth}
	\includegraphics[height=0.88\textwidth]{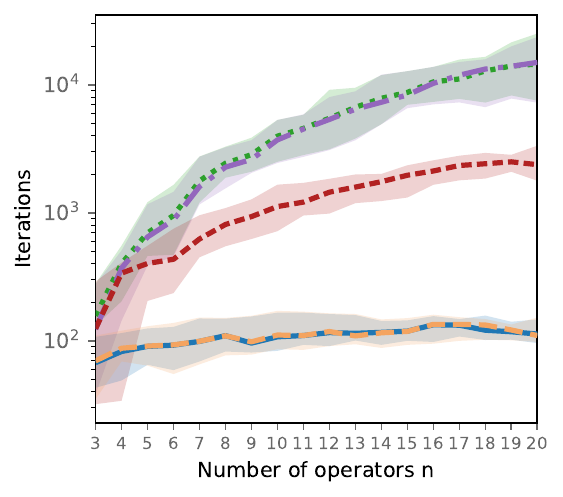}
	\end{subfigure}\hfill\begin{subfigure}{\ancho\textwidth}
	\includegraphics[height=.88\textwidth]{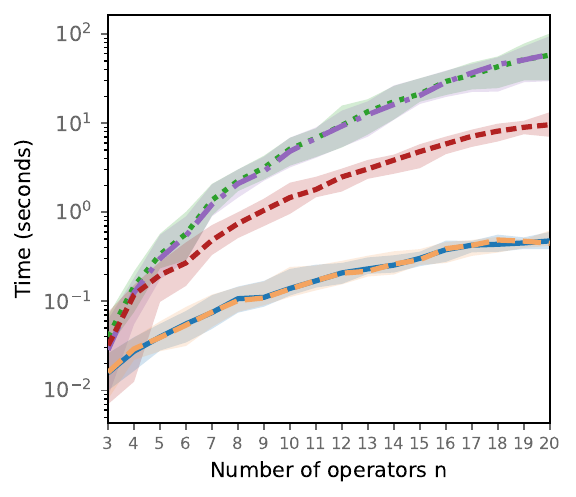}
	\end{subfigure}
	\caption{Results of the numerical experiment comparing five graph configurations for different number of operators}
	\label{fig: experiment}
\end{figure}

As we can observe, the behavior of the five algorithms can be grouped into three categories. The slowest one is formed by the ring and the sequential methods. Both use a sequential graph for $G^\prime$ and reached the tolerance practically at the same time. The intermediate category is comprised by parallel, which takes $G^\prime$ as the parallel graph. Finally, the group formed by the two complete graphs was the fastest among all of them. Both use $G^\prime$ as the complete graph and practically converged with identical speed.
 
Our experiment corroborates what was observed in~\cite{graph-drs}: a key factor for the performance of the algorithm relies on the choice of $G^\prime$. The complete forward-backward method, which is the algorithm with more connections, was the fastest. Although the number of edges seems to influence the speed of convergence, it is not the sole factor, as both parallel and sequential graphs are trees with $n-1$ edges, but parallel was significantly faster. Likely, the \emph{algebraic connectivity} of the subgraph $G^\prime$ (the smallest nonzero eigenvalue of its Laplacian) plays an important role in the performance, as it was noticed in~\cite{graph-drs}. Indeed, observe that the complete graph has algebraic connectivity $n$, the parallel graph has value $1$, and the ring and sequential graphs have $2(1-\cos(\pi/n))$ \cite{algconectivity}. This seems to be a common phenomenon in consensus algorithms (see, e.g.,~\cite{olfati2004consensus,olfati2007consensus,consensus22}).

\section{Concluding remarks}\label{sec:conclusion}

In this work we have introduced a unifying framework to construct frugal splitting forward-backward algorithms with minimal lifting for finding a zero in the sum of finitely many maximally monotone operators. This approach is an extension of  that of \cite{graph-drs} to include cocoercive operators, which are evaluated through forward steps.

Different algorithms can be constructed by imposing distinct connection patterns among the variables defining the scheme, which are modeled by certain graphs. This permits to recover some known methods in the literature, as well as derive new ones. The advantage of this framework when compared with the ad hoc technical convergence proofs designed for each particular algorithm in~\cite{ring-networks,backforw,davisyin,malitsky2023resolvent,genFB,ryu20,4operator} is clear.

As a by-product, we have derived a new splitting algorithm configured with complete graph information which significantly outperformed existing methods in our numerical test. Although this is far from an exhaustive computational study, the promising results encourage us to further investigate this algorithm in future research.

Lastly, the connections with~\cite{MBG24} are intriguing. We leave as an open question the development of an extension allowing to cover both settings, as well as the study of the role of the algebraic connectivity in the performance of these algorithms.

\appendix
\section{Appendix}

\begin{lemma}\label{lem: ai ti ai+1}
    For all $i\in\llbracket 1,n-2\rrbracket$, it holds $a_i^2=t_i^2+a_{i+1}^2$ and $a_{n-1}=-t_{n-1}=\sqrt{\frac{n}{2}}$, where $a_i$ and $t_i$ are given by~\eqref{eq: a_i t_i}.
\end{lemma}

\begin{proof}
    Clearly $a_{n-1}=-t_{n-1}$, since
    \[-t_{n-1}=\sqrt{\frac{n}{(n-(n-1))(n-(n-1)+1)}}=\sqrt{\frac{n}{2}}=\sqrt{\frac{n-(n-1)n}{n-(n-1)+1}}=a_{n-1}.\]
    Let us prove now the equality $a_i^2=t_i^2+a_{i+1}^2$. Starting by the right-hand side of the equation, one gets that
    \begin{align*}
        t_i^2+a_{i+1}^2&=\frac{n}{(n-i)(n-i+1)}+\frac{(n-i-1)n}{n-i}=\frac{n+(n-i-1)(n-i+1)n}{(n-i)(n-i+1)}\\
        &=\frac{n(1+(n-i)^2-1)}{(n-i)(n-i+1)}
        =\frac{n(n-i)}{(n-i+1)}
        =a_i^2,
    \end{align*}
    which concludes the proof.
\end{proof}

\begin{proposition}\label{prop: onto complete}
    The matrix $Z\in\R^{n\times(n-1)}$ defined in~\eqref{eq: Z complete} satisfies the following:
    \begin{enumerate}[label=(\roman*)]
        \item $\rk Z=n-1$,
        \item $L=ZZ^\ast$, where $L$ is the Laplacian matrix of the complete graph given in~\eqref{eq: Lap complete}.
    \end{enumerate}
\end{proposition}

\begin{proof}
    First of all, notice that the matrix is lower triangular with nonzero entries in its diagonal. Hence, it has maximal rank, i.e., $\rk Z=n-1$.

    To prove assertion~\emph{(ii)}, let  $Z_i=(t_1,\ldots,t_{i-1},a_i,0,\ldots,0)\in\R^{n-1}$ be the $i$-th row of $Z$. Since
    \[(ZZ^\ast)_{ij}=\begin{cases}
        \norm{Z_i}^2 & \text{if }i=j,\\
        \langle Z_i,Z_j\rangle & \text{otherwise},
    \end{cases}\]
    we need to show that $\norm{Z_i}^2=n-1$ and $\langle Z_i,Z_j\rangle=-1$ for all $i\neq j$.

    Let us first prove by induction that $\norm{Z_i}^2=n-1$, for all $i=1,\ldots,n$. By how  $Z_1$ is defined, we get that $\norm{Z_1}^2=a_1^2=n-1$. Now, suppose that $\norm{Z_i}^2=n-1$ for some $i\geq 1$. Hence, by the structure of $Z$ and Lemma~\ref{lem: ai ti ai+1}, we get that
    \[\norm{Z_{i+1}}^2=\norm{Z_i}^2-a_i^2+t_i^2+a_{i+1}^2=\norm{Z_i}^2=n-1.\]
    This shows that $\norm{Z_i}^2=n-1$, for all $i=1,\ldots,n-1$. However, notice that, again by Lema~\ref{lem: ai ti ai+1}, $a_{n-1}=-t_{n-1}$. Thus, $\norm{Z_n}^2=\norm{Z_{n-1}}^2=n-1$.

    Now, we show that $\langle Z_i,Z_j\rangle=-1$ for all $i\neq j$. By how the vectors $Z_i$ are defined, one can verify that
    \[\langle Z_i,Z_j\rangle=\langle Z_i,Z_{i+1}\rangle=\sum_{k=1}^{i-1}t_k^2+a_it_i,\quad\forall i=1,\ldots,n-1,\forall j>i.\]
    By symmetry of the dot product, this shows that the problem is reduced to the case $\langle Z_i,Z_{i+1}\rangle$ for all $i=1,\ldots,n-1$.
    Since $\norm{Z_i}^2=\sum_{k=1}^{i-1}t_k^2+a_i^2$, then
    \[\langle Z_i,Z_{i+1}\rangle=\norm{Z_i}^2-a_i^2+a_it_i.\]
    By definition of $a_i$ and $t_i$, we have that $a_it_i=-\frac{n}{n-i+1}$. Hence,
    \begin{align*}
        \norm{Z_i}^2-a_i^2+a_it_i&=n-1-\frac{(n-i)n}{n-i+1}-\frac{n}{n-i+1}=-1.
    \end{align*}
    Since this is true for all $i=1,\ldots,n-1$, all cases has been proven and, as a consequence, $L=ZZ^\ast$.
\end{proof}
\section*{Declarations}


\paragraph{Conflict of interest} The authors declare no competing interests.

\printbibliography

@article{graph-drs,
  author     = {Bredies, Kristian and Chenchene, Enis and Naldi, Emanuele},
  title      = {Graph and Distributed Extensions of the {D}ouglas--{R}achford Method},
  journal    = {SIAM Journal on Optimization},
  year       = {2024},
  volume     = {34},
  number     = {2},
  pages      = {1569-1594},
  doi        = {10.1137/22M1535097}
}

@article{degenerate-ppp,
  author   = {Bredies, Kristian and Chenchene, Enis and Lorenz, Dirk A. and Naldi, Emanuele},
  title    = {Degenerate Preconditioned Proximal Point Algorithms},
  journal  = {SIAM Journal on Optimization},
  year     = {2022},
  volume   = {32},
  number   = {3},
  pages    = {2376-2401},
  doi       = {10.1137/21M1448112}
}

@article{ring-networks,
	author = {Aragón-Artacho, Francisco J. and Malitsky, Yura and Tam, Matthew K. and Torregrosa-Belén, David},
	title = {Distributed forward-backward methods for ring networks},
	journal = {Computational Optimization and Applications},
	year = {2023},
	volume = {86},
	number = {3},
	pages = {845--870},
	doi = {10.1007/s10589-022-00400-z}
}

@book{godsil,
  author = {Godsil, Chris and Royle, Gordon F.},
  title = {Algebraic Graph Theory},
  year = {2001},
  series = {Graduate Texts in Mathematics},
%  number = {Book 207},
  publisher = {Springer New York, NY},
  doi = {10.1007/978-1-4613-0163-9}
}

@book{bauschke,
  author = {Bauschke, Heinz H. and Combettes, Patrick L.},
  title = {Convex Analysis and Monotone Operator Theory in Hilbert Spaces},
  year = {2017},
  series = {CMS Books in Mathematics},
  edition = {2nd},
  publisher = {Springer Cham},
  doi = {10.1007/978-3-319-48311-5}
}

@article{minty,
  author={George J. Minty},
  title={Monotone (nonlinear) operators in {H}ilbert space},
  journal={Duke Mathematical Journal},
  year={1962},
  volume={29},
  number={3},
  pages={341-346},
  doi={10.1215/S0012-7094-62-02933-2}
}

@article{RockaProx,
  author={Rockafellar, R. Tyrrell},
  title={Monotone operators and the proximal point algorithm},
  journal={SIAM Journal on Control and Optimization},
  year={1976},
  volume={14},
  number={5},
  pages={877--898},
  doi = {10.1137/0314056}
}

@article{davisyin,
  author={Davis, Damek and Yin, Wotao},
  title={A three-operator splitting scheme and its optimization applications},
  journal={Set-Valued and Variational Analysis},
  year={2017},
  volume={25},
  number={4},
  pages={829--858},
  doi = {10.1007/s11228-017-0421-z}
}

@article{backforw,
  title={Backward--forward algorithms for structured monotone inclusions in Hilbert spaces},
  author={Attouch, Hedy and Peypouquet, Juan and Redont, Patrick},
  journal={Journal of Mathematical Analysis and Applications},
  year={2018},
  volume={457},
  number={2},
  pages={1095--1117},
  publisher={Elsevier},
  doi = {10.1016/j.jmaa.2016.06.025}
}

@article{LM79,
  author={Lions, Pierre-Louis and Mercier, Bertrand},
  title={Splitting algorithms for the sum of two nonlinear operators},
  journal={SIAM Journal on Numerical Analysis},
  year={1979},
  volume={16},
  number={6},
  pages={964--979},
  doi={10.1137/0716071}
}

@article{genFB,
  author={Raguet, Hugo and Fadili, Jalal and Peyr{\'e}, Gabriel},
  title={A generalized forward-backward splitting},
  journal={SIAM Journal on Imaging Sciences},
  year={2013},
  volume={6},
  number={3},
  pages={1199--1226},
  doi={10.1137/120872802}
}

@article{pierra,
  author={Pierra, Guy},
  title={Decomposition through formalization in a product space},
  journal={Mathematical Programming},
  year={1984},
  volume={28},
  number={1},
  pages={96--115},
  doi={10.1007/BF02612715}
}

@article{ryu20,
  author={Ryu, Ernest K},
  title={Uniqueness of {DRS} as the 2 operator resolvent-splitting and impossibility of 3 operator resolvent-splitting},
  journal={Mathematical Programming},
  year={2020},
  volume={182},
  number={1},
  pages={233--273},
  doi={10.1007/s10107-019-01403-1}
}

@article{malitsky2023resolvent,
  author={Malitsky, Yura and Tam, Matthew K},
  title={Resolvent splitting for sums of monotone operators with minimal lifting},
  journal={Mathematical Programming},
  year={2023},
  volume={201},
  number={1},
  pages={231--262},
  doi={10.1007/s10107-022-01906-4}
}

@Article{MBG24,
  author   = {Morin, Martin and Banert, Sebastian and Giselsson, Pontus},
  title    = {Frugal Splitting Operators: Representation, Minimal Lifting, and Convergence},
  journal  = {SIAM Journal on Optimization},
  year     = {2024},
  volume   = {34},
  number   = {2},
  pages    = {1595-1621},
  doi      = {10.1137/22M1531105}
}

@Article{4operator,
  author   = {Zong, Chunxiang and Tang, Yuchao and Zhang, Guofeng},
  title    = {Solving monotone inclusions involving the sum of three maximally monotone operators and a cocoercive operator with applications},
  journal  = {Set-Valued and Variational Analysis},
  year     = {2023},
  volume   = {31},
  number   = {2},
  pages    = {16},
  doi      = {10.1007/s11228-023-00677-0}
}

@Article{condat2023proximal,
  author={Condat, Laurent and Kitahara, Daichi and Contreras, Andr{\'e}s and Hirabayashi, Akira},
  title={Proximal splitting algorithms for convex optimization: A tour of recent advances, with new twists},
  journal={SIAM Review},
  year={2023},
  volume={65},
  number={2},
  pages={375--435},
  doi={10.1137/20M1379344}
}

@Article{tam2023frugal,
  author={Tam, Matthew K},
  title={Frugal and decentralised resolvent splittings defined by nonexpansive operators},
  journal={Optimization Letters},
  year={2023},
  pages={1--19},
  doi={10.1007/s11590-023-02064-y}
}

@Article{algconectivity,
  author={De Abreu, Nair Maria Maia},
  title={Old and new results on algebraic connectivity of graphs},
  journal={Linear Algebra and its Applications},
  year={2007},
  volume={423},
  number={1},
  pages={53--73},
  doi={10.1016/j.laa.2006.08.017}
}

@article{olfati2004consensus,
  author={Olfati-Saber, R. and Murray, R.M.},
  title={Consensus problems in networks of agents with switching topology and time-delays},
  journal={IEEE Transactions on Automatic Control}, 
  year={2004},
  volume={49},
  number={9},
  pages={1520-1533},
  doi={10.1109/TAC.2004.834113}
}

@article{olfati2007consensus,
  author={Olfati-Saber, Reza and Fax, J. Alex and Murray, Richard M.},
  title={Consensus and Cooperation in Networked Multi-Agent Systems}, 
  journal={Proceedings of the IEEE}, 
  year={2007},
  volume={95},
  number={1},
  pages={215-233},
  doi={10.1109/JPROC.2006.887293}
}

@article{consensus22,
  title={Consensus in multi-agent systems: a review},
  author={Amirkhani, Abdollah and Barshooi, Amir Hossein},
  journal={Artificial Intelligence Review},
  year={2022},
  volume={55},
  number={5},
  pages={3897--3935},
  doi={10.1007/s10462-021-10097-x}
}

\end{document}